\documentclass{article}
\usepackage{amsmath,amsfonts,amsthm,graphics,color}
\usepackage{accents} 

\newtheorem{theorem}{Theorem}[section]
\newtheorem{lemma}[theorem]{Lemma}

\newtheorem{corollary}[theorem]{Corollary}
\newtheorem{conjecture}{Conjecture}
\theoremstyle{definition}

\newtheorem{remark}[theorem]{Remark}

\newtheorem*{ack}{Acknowledgement}

\newcounter{thmenumerate}
\newcommand{\aex}{a.e.\spacefactor=1000}
\newcommand\noproof{\qed}   

\newcommand\Tk{T_{\ka}}
\newcommand\bp{{\mathfrak X}}

\newcommand\eps{\varepsilon}
\newcommand\sss{{\mathcal S}}
\newcommand\dd{\,d}
\newcommand\norm[1]{\ensuremath{\|#1\|}}
\newcommand\tn[1]{\ensuremath{\|#1\|_2}}
\newcommand\bb[1]{\bigl(#1\bigr)}

\newcommand\upto{\nearrow}

\newcommand{\refT}[1]{Theorem~\ref{#1}}
\newcommand{\refR}[1]{Remark~\ref{#1}}
\newcommand{\refC}[1]{Corollary~\ref{#1}}
\newcommand{\refL}[1]{Lemma~\ref{#1}}
\newcommand{\refS}[1]{Section~\ref{#1}}
\newcommand{\refSS}[1]{Subsection~\ref{#1}}

\newcommand\E{{\mathop{\mathbb E{}}\nolimits}}
\renewcommand\Pr{{\mathop{\mathbb P{}}\nolimits}}
\newcommand\ka{\kappa}
\newcommand\la{\lambda}

\newcommand\F{{\mathcal F}}
\newcommand\T{{\mathcal T}}

\newcommand\bpk{{\mathfrak X}_{\ka}}

\newcommand\kae{{\ka_{\mathrm e}}}

\newcommand\kf{{\undertilde{\ka}}}

\newcommand\pto{\overset{\mathrm{p}}{\to}}

\newcommand\aut{\operatorname{aut}}
\newcommand\cc{{\mathrm{c}}}
\newcommand\dcut{{\delta_{\square}}}
\newcommand\dcutone{{\delta_{\square,1}}}
\newcommand\dcuttwo{{\delta_{\square,2}}}

\newcommand\set[1]{\ensuremath{\{#1\}}}

\newcommand\bigpar[1]{\bigl(#1\bigr)}
\newcommand\Bigpar[1]{\Bigl(#1\Bigr)}

\newcommand\bigabs[1]{\bigl|#1\bigr|}
\newcommand\Bigabs[1]{\Bigl|#1\Bigr|}

\newcommand\lrabs[1]{\left|#1\right|}

\newcommand\oi{[0,1]}

\newcommand\floor[1]{\lfloor #1\rfloor}

\newcommand\ntoo{\ensuremath{{n\to\infty}}}
\newcommand\mtoo{\ensuremath{{m\to\infty}}}

\newcommand\cF{\mathcal F}

\newcommand\cW{\mathcal W}

\newcommand\ga{\alpha}

\newcommand\Var{\operatorname{Var}}

\newcommand\ett[1]{\boldsymbol1[#1]} 

\newcommand\qw{^{-1}}

\newcommand\op{o_{\mathrm{p}}}

\newcommand\gl{\lambda}
\newcommand\go{\omega}

\newcommand\nkt{N_k^{\mathsf t}}
\newcommand\nkc{N_k^{\mathsf c}}

\renewcommand{\=}{:=}
\newcommand\bignorm[1]{\bigl\|#1\bigr\|}
\newcommand\Bignorm[1]{\Bigl\|#1\Bigr\|}

\newcommand\ooo{[0,\infty)}
\newcommand\rplus{[0,\infty)}

\newcommand\glw{\gl_W}
\newcommand{\tx}{t_{\mathrm{isol}}}
\newcommand{\ty}{t_{0}}
\newcommand{\tz}{t_{1}}
\newcommand\cn[1]{\norm{#1}\cut}
\newcommand\cntwo[1]{\norm{#1}_{\square,2}}
\newcommand\cnone[1]{\norm{#1}_{\square,1}}
\newcommand\cutnorm{\cn}
\newcommand\cut{_{\square}}
\newcommand\sij{_{ij}}
\newcommand\ab{^{(a,b)}}
\newcommand\wab{W\ab}
\newcommand{\cw}{\cW}
\newcommand{\cwsR}{\cW_{\mathrm{sym}}^{(R)}}
\newcommand{\cws}{\cW_{\mathrm{sym}}}
\newcommand{\cwa}{\cW_{\BB}}
\newcommand{\cwrA}{\cW_{r,\BB}}
\newcommand\liss{_{L^1(\sss^2)}}
\newcommand\lis{_{L^1(\sss)}}

\newcommand\kam{\ka^-}
\newcommand\emi{e^-}

\newcommand\vxs{{\mathcal V}}
\newcommand\xss{({\bf x}_n)}
\newcommand\xs{{\bf x}}
\newcommand\whp{whp}
\newcommand\on[1]{\norm{#1}_{L^1}}
\newcommand\onb[1]{\Bignorm{#1}_{L^1}}
\newcommand\sn[1]{\norm{#1}_{\infty}}

\newcommand\rn{\tau_n}
\newcommand\karn{{\ka^{(\rn)}}}
\newcommand\bv{{\bf v}}
\newcommand\isom{\cong}
\newcommand\cwr{\cW_r}
\newcommand\cwrs{\cW_{r,\mathrm{sym}}}
\newcommand\uW{\undertilde{W}}
\newcommand\ta{{(\tau)}}
\newcommand\hir{h_{i_1i_2\ldots i_r}}
\newcommand\hW{\widehat W}
\newcommand\hka{\widehat\kappa}
\newcommand\setdiff{\bigtriangleup}
\newcommand\ba{{\bf a}}
\newcommand\bw{{\bf w}}
\newcommand\tW{{\widetilde W}}
\newcommand\dt{{\mathrm d}t}
\newcommand\HH{{\mathcal H}}
\newcommand\BB{B}
\newcommand\Gp{G_{\mathrm{Po}}}
\newcommand\Gpm{G_{\mathrm{Po}}^{\mathrm{m}}}
\newcommand\VW{V}
\newcommand\vw{v}
\newcommand\wst{\overset{\mathrm{w*}}{\longrightarrow}}
\newcommand\rhos{{\rho^*}}
\newcommand\iid{i.i.d.}

\begin{document}
\title{The cut metric, random graphs, and branching processes}

\author{B\'ela Bollob\'as\thanks{Department of Mathematical Sciences,
University of Memphis, Memphis TN 38152, USA}
\thanks{Trinity College, Cambridge CB2 1TQ, UK}
\thanks{Research supported in part by NSF grants DMS-0906634,
 CNS-0721983 and CCF-0728928, and ARO grant W911NF-06-1-0076}
\and Svante Janson%
\thanks{Department of Mathematics, Uppsala University,
 PO Box 480, SE-751 06 Uppsala, Sweden}
\and Oliver Riordan%
\thanks{Mathematical Institute, University of Oxford, 24--29 St Giles', Oxford OX1 3LB, UK}}
\date{January 14, 2009; revised February 06, 2010}
\maketitle

\begin{abstract}

In this paper we study the component structure of random graphs with
independence between the edges.  Under mild assumptions, we determine
whether there is a giant component, and find its asymptotic size when it  
exists.  We assume that the sequence of matrices of edge probabilities
converges to an appropriate limit object (a kernel), but only in a very
weak sense, namely in the cut metric.  Our results thus generalize
previous results on the phase transition in the already very general  
inhomogeneous random graph model introduced by the present authors
in~\cite{kernels}, as well as related results of Bollob\'as, Borgs, Chayes
and Riordan~\cite{QRperc}, all of which involve considerably stronger
assumptions. We also prove corresponding results for random hypergraphs;    
these generalize our results on the phase transition in inhomogeneous   
random graphs with clustering~\cite{clustering}.

\end{abstract}

\section{Introduction and results}

Throughout this paper we consider random graphs with independence
between the edges. The distribution of a random $n$-vertex graph with
this property
is of course specified by the matrix of edge probabilities; here we
are interested in the asymptotic behaviour of the component structure
as $n\to\infty$, so we shall consider a sequence of such matrices.
Our main focus is to determine when there is whp a giant component, i.e., a
component containing $\Theta(n)$ vertices. Here, as usual,
an event holds {\em with high probability}, or {\em whp}, if it holds with
probability $1-o(1)$ as $n\to\infty$.
When there is a giant component, we shall also find its asymptotic size.

For these questions it is natural to focus on (extremely)
{\em sparse} graphs, with $\Theta(n)$ edges, so we shall
normalize by considering matrices $A_n$ whose entries are
$n$ times the corresponding edge probabilities. Thus the case in which
each $A_n$ has all (off-diagonal) entries equal to some $c>0$
corresponds to the classical sparse model $G(n,c/n)$.
Without some further assumptions, it seems difficult
to prove asymptotic results, although Alon \cite{Alon} did so 
for some questions concerning connectedness.
As in previous work, the natural additional assumption turns
out to be convergence to a suitable limiting object,
namely a {\em kernel}, i.e., a symmetric non-negative
function on $[0,1]^2$. Our aim is to relate the asymptotic
size of the giant component to a suitable function
of this kernel.

The aim described above was also one of the aims of~\cite{kernels},
and of Bollob\'as, Borgs, Chayes and Riordan~\cite{QRperc}.
We shall prove a common generalization of
the corresponding results from these papers by weakening
the assumptions: we shall work with convergence in the cut metric 
(defined below) as in~\cite{QRperc}, while allowing
unbounded matrices and kernels, as in~\cite{kernels}.
It turns out that these very weak, natural assumptions suffice
to allow us to relate the giant component
of the random graph to the kernel.

To state our results we shall need a few definitions.
By a {\em kernel} on $[0,1]$ we simply mean an integrable,
symmetric function $\ka:[0,1]^2\to \rplus$. We regard kernels as elements of $L^1$,
so two kernels that are equal almost everywhere are considered to be the same.

Throughout, $A_n$ will denote a symmetric $n$-by-$n$ matrix with non-negative entries.
If $A_n=(a\sij)$ is such a matrix, then there
is a piecewise constant kernel $\ka_{A_n}$ naturally associated to $A_n$:
this takes the value $a_{ij}$ on the square $((i-1)/n,i/n]\times ((j-1)/n,j/n]$.
We call $\ka$ an {\em $n$-by-$n$ kernel} if it is of the form $\ka_{A_n}$ for
some $A_n$.

There is a (sparse) random graph naturally associated to $A_n$, namely the graph
$G(A_n)=G_{1/n}(n,A_n)$. This graph has vertex set $[n]=\{1,2,\ldots,n\}$,
the events that different edges are present are independent,
and the probability that $ij$ is present is $\min\{a_{ij}/n,1\}$.
If some of the $a_{ii}$ are non-zero then $G(A_n)$ may contain loops; this will be irrelevant
for us here, since we study only the component structure of $G(A_n)$.
Often, it is convenient to consider minor variants of these definitions:
in the {\em Poisson multi-graph} variant, $\Gpm(A_n)$, the number of copies
of each possible edge $ij$ is Poisson with mean $a_{ij}/n$. In the {\em Poisson
simple graph} variant, $\Gp(A_n)$, the probability that $ij$ is present is $1-\exp(-a_{ij}/n)$;
in both cases the numbers of copies of different edges are independent.
Thus $\Gp(A_n)$ is the simple graph underlying $\Gpm(A_n)$. Most of the time
it makes no difference which variant we consider. Indeed, 
whenever $a_{ij}<n/2$, say, for all $i$ and $j$, then
\begin{equation}\label{convert}
 G(A_n) =_{\mathrm{d}} \Gp(A_n')
\end{equation}
where $=_{\mathrm{d}}$ denotes equality in distribution, and $A_n'$ is the matrix with entries
\begin{equation}\label{cve}
 a_{ij}' = - n\log(1-a_{ij}/n) = a_{ij} +O(a_{ij}^2/n).
\end{equation}
In the typical case considered here, the entries $a_{ij}$ are small
compared to $n$, 
so switching
between $G(\cdot)$ and $\Gp(\cdot)$ thus corresponds to a minor change in
the edge probability parameters.
Moreover, under the rather weak assumptions $\max_{ij} a_{ij}<n/2$ and
$\sum_{i,j=1}^na_{ij}^3=o(n^3)$, the random graphs 
$G(A_n)$ and $\Gp(A_n)$ are asymptotically equivalent in the strong
sense that they can be coupled so that they are equal \whp;
see \cite[Corollary 2.13]{SJ212}.

Having described the limit object (a kernel), and the random graph,
it remains to describe the notion of convergence. In doing so
it is convenient to consider somewhat more general kernels.

Let $(\sss,\mu)$ be a probability space; most of the time we shall
take $\sss$ to be $[0,1]$ (or $(0,1]$) with $\mu$ Lebesgue measure.
A {\em kernel} on $\sss$ is
an integrable, symmetric function $\ka:\sss^2\to \rplus$. 
Following Frieze and Kannan~\cite{FKquick}, for $W\in L^1(\sss^2)$ we define
the {\em cut norm} $\cn{W}$ of $W$ by
\begin{equation}\label{cnodef}
 \cnone W \= \sup_{S,T}
  \Bigabs{\int_{S \times T} W(x,y) \dd\mu(x)\dd\mu(y)},
\end{equation}
where the supremum is taken over all pairs of measurable subsets of $\sss$.
Alternatively, one can take
\begin{equation}\label{cntdef}
 \cntwo W \= \sup_{\sn{f},\sn{g}\le1}
  \Bigabs{\int_{\sss^2} f(x)W(x,y)g(y)\dd\mu(x)\dd\mu(y)}.
\end{equation}
In taking the supremum in \eqref{cntdef} one can restrict to functions
$f$ and $g$ taking only the values $\pm 1$; it follows that
\[
 \cnone{W} \le \cntwo{W} \le 4\cnone{W}.
\]
Thus the two norms $\cnone{\cdot}$ and $\cntwo{\cdot}$ are equivalent,
and it will almost never matter which one we use.
We shall write $\cn{\cdot}$ for either norm, commenting
in the few cases where the choice matters.
(There are further, equivalent versions of the cut-norm; see
Borgs, Chayes, Lov\'asz, S\'os and Vesztergombi~\cite{BCLSV:1}.)

Note that for either definition of the cut norm we have
\[
 \Bigabs{\int W} \le \cn{W} \le \on{W}.
\]

The definition \eqref{cntdef} is natural for a functional analyst:
this norm is the dual of the projective tensor product norm in
$L^\infty\hat\otimes L^\infty$, and is thus the injective tensor
product norm in $L^1\check\otimes L^1$; equivalently, it is equal to the 
operator norm of the corresponding integral operator $L^\infty\to
L^1$.
One advantage of this version
is the simple ``Banach module'' property we shall note later in \eqref{cutnorm2}.
On the other hand, \eqref{cnodef} is probably more familiar in combinatorics,
and (surprisingly) occasionally has a tiny advantage; see \refS{sec_hyp}.

Given a kernel $\ka$ and a measure-preserving bijection $\tau:\sss\to\sss$,
let $\ka^\ta$ be the kernel defined by
\[ 
 \ka^\ta(x,y) = \ka(\tau(x),\tau(y));
\]
we call $\ka^\ta$ a {\em rearrangement} of $\ka$. We write $\ka\sim \ka'$ if $\ka'$ is a rearrangement
of $\ka$.
Given two kernels $\ka$, $\ka'$ on $[0,1]$,
the {\em cut metric} of
Borgs, Chayes, Lov\'asz, S\'os and Vesztergombi~\cite{BCLSV:1}
is defined by
\begin{equation}\label{cutdef}
 \dcut(\ka,\ka') = \inf_{\ka''\sim \ka'} \cn{\ka-\ka''}.
\end{equation}
If we wish to specify which version of the cut norm is involved, we write $\dcutone$
or $\dcuttwo$. Usually, this is irrelevant. 

As in~\cite{BCLSV:1}, one
can also define $\dcut$ using couplings between different kernels, rather
than rearrangements. In this case it is irrelevant that the kernels
are on the same probability space.
In particular, we may regard a matrix $A_n$ as a kernel on the
discrete space with $n$ equiprobable elements. Then (by an obvious coupling)
$\dcut(A_n,\ka_{A_n})=0$, where $\ka_{A_n}$ is the $n$-by-$n$ kernel
on $[0,1]$ corresponding to $A_n$. Thus 
$\dcut(A_n,\ka)=\dcut(\ka_{A_n},\ka)$ for any kernel $\ka$ on any
probability space $(\sss,\mu)$.
In the light of this we shall often
identify a matrix with the corresponding kernel on $[0,1]$.

Throughout this paper, we shall consider sequences $(A_n)$ of matrices such that
for some kernel $\ka$ we have
$\dcut(A_n,\ka)\to 0$. 
It follows from the results of~\cite{SJ210} that for any kernel $\ka$ on a
probability space $(\sss,\mu)$, there exists a kernel $\ka'$ on $\oi$
with $\dcut(\ka,\ka')=0$.
Hence we lose no generality by taking $(\sss,\mu)$ to be the 
{\em standard ground space} in which $\sss=[0,1]$ (or $(0,1])$
and $\mu$ is  Lebesgue measure.
In this case
it is natural to identify $A_n$ with $\ka_{A_n}$ as above, and we may
use the more down-to-earth formula \eqref{cutdef} as the definition of $\dcut$.

To state our results we need two further definitions, from \cite{kernels}.
Given a kernel $\ka$ on a probability space $(\sss,\mu)$,
let $\bpk$ be the multi-type Galton--Watson
branching process defined as follows.
We start with a single particle in generation $0$, whose type has the distribution
$\mu$. A particle in generation $t$ of type $x$ gives rise to children in generation
$t+1$ whose types form a Poisson process on $\sss$ with intensity $\ka(x,y)\dd\mu(y)$.
The children of different particles are independent, and independent of the history.

We shall also consider the branching processes $\bpk(x)$, $x\in \sss$,
defined  as above except that $\bpk(x)$ starts with a
single particle of the given type $x$.

Let $\rho(\ka)$ denote the {\em survival probability} of $\bpk$, i.e., the probability
that all generations are non-empty. 
It is easily seen that this is the same
as the probability that the total number $|\bpk|$ of particles in $\bpk$ is infinite.
For basic results about $\rho(\ka)$, we refer the reader to~\cite{kernels}.

Finally, as in~\cite{kernels}, a kernel $\ka$ is {\em reducible}
if there exists $A\subset \sss$ with $0<\mu(A)<1$ such that $\ka$
is zero almost everywhere on $A\times (\sss\setminus A)$.
Otherwise, $\ka$ is {\em irreducible}.

Throughout, we use standard notation for probabilistic asymptotics
as in~\cite{JLR}. For example, $\pto$ denotes convergence in probability,
and $X_n=\op(f(n))$ means $X_n/f(n)\pto 0$.

\subsection{Main results}\label{ss_res}

In this subsection we state our main results; we shall give corresponding
results for hypergraphs in Section~\ref{sec_hyp}.
Recall that any matrix denoted by $A_n$
is assumed to be a symmetric $n$-by-$n$ matrix with non-negative entries.
Given a graph $G$ and an $i\ge 1$, we write $C_i(G)$ for the number of vertices
in the $i$th largest component of $G$, with $C_i(G)=0$ if $G$ has fewer
then $i$ components.
We shall see later that our results imply corresponding results for the Poisson
variants of $G(A_n)$; for simplicity we state them only in the original formulation,
where the edge probabilities are $\min\{a_{ij}/n,1\}$. 
The theorems are valid for a kernel $\ka$ on any probability space
$(\sss,\mu)$, but as noted above we may assume without loss of generality
that $\sss=\oi$, and we shall do so in the proofs for
convenience.

\begin{theorem}\label{th1}
Let $\ka$ be a kernel and $(A_n)$
a sequence of symmetric non-negative $n$-by-$n$ matrices such that
$\dcut(A_n,\ka)\to 0$.
Then $C_1(G(A_n))/n \le  \rho(\ka)+\op(1)$. If $\ka$ is irreducible, then
$C_1(G(A_n))/n\pto \rho(\ka)$ and $C_2(G(A_n))=\op(n)$. 
\end{theorem}

Of course, as usual
we do not require $A_n$ to be defined for every $n$, only for a subsequence.

Let $\rho_\ka(x)$ denote the survival probability of the process $\bpk(x)$
started with a particle of type $x$. Let $\Tk$ be the
integral operator on $\sss$ with kernel $\ka$,
defined by
\begin{equation}  \label{tk}
(\Tk f)(x)=\int_\sss \ka(x,y)f(y)\dd\mu(y),
\end{equation}
for any (measurable) function $f$ such that this integral is defined (finite or
$+\infty$) for \aex{} $x$.
Note that this class of functions includes every (measurable) function $f\ge 0$.
Also, let
\[
 \norm{\Tk} = \sup\bigl\{\tn{\Tk f}: \tn{f}\le1,\, f\ge 0\bigr\} \le\infty;
\]
clearly if $\norm{\Tk}<\infty$, then
$\norm{\Tk}$ is simply the norm of $\Tk$ as an operator on $L^2(\sss,\mu)$.

Recall from~\cite[Theorem 6.2]{kernels} that $\rho(\ka)>0$ if and only
if $\norm{\Tk}>1$,
and that if $\norm{\Tk}>1$, then $\rho_\ka$ is the unique non-zero solution $f\ge 0$
to the functional equation
\[
 f = 1-\exp(-\Tk f).
\]

Using \refT{th1}, we shall deduce the following slight extension,
describing the `critical' value of $c$ above which a giant component
appears in $G(c A_n)$.

\begin{theorem}\label{th2}
Let $\ka$ be a kernel, $(A_n)$ a sequence of symmetric non-negative $n$-by-$n$ matrices such that
$\dcut(A_n,\ka)\to 0$, and $c>0$ a constant, and set $G_n=G(cA_n)$.
\begin{enumerate}
 \renewcommand{\labelenumi}{\textup{(\alph{enumi})}}%
 \renewcommand{\theenumi}{\textup{(\alph{enumi})}}%

\item
 If $c\le \norm{\Tk}^{-1}$, then $C_1(G_n)=\op(n)$.

\item 
If $c>\norm{\Tk}^{-1}$, then $C_1(G_n)=\Theta(n)$ \whp.
Furthermore, if $\ka$ is bounded,
then for any constant $\alpha<(c\norm{\Tk}-1)/(c\sup\ka)$
we have $C_1(G_n)\ge \alpha n$ \whp.
\item 
If $\ka$ is irreducible, then $C_1(G_n)/n\pto \rho(c\ka)$
and $C_2(G_n)=\op(n)$.
\end{enumerate}
\end{theorem}

This clearly generalizes the main result, Theorem 1,
of Bollob\'as, Borgs, Chayes and Riordan~\cite{QRperc}, which is simply the
special case in which $\ka$ and the entries of the matrices $A_n$ are uniformly bounded.
As we shall see in the next subsection, \refT{th2} also generalizes Theorem 3.1 of~\cite{kernels}.
Note, however, that to prove this requires various results from \cite{kernels}.

Returning to the irreducible case, we shall also prove a `stability'
result analogous to Theorem 3.9 of~\cite{kernels}.

\begin{theorem}\label{th_stab}
Let $\ka$ be an irreducible kernel and $(A_n)$
a sequence of non-negative symmetric $n$-by-$n$ matrices such that
$\dcut(A_n,\ka)\to 0$.
For every $\eps>0$ there is a $\delta=\delta(\ka,\eps)>0$ such that, whp,
\[
 \rho(\ka)-\eps \le C_1(G_n')/n \le \rho(\ka)+\eps
\]
for every graph $G_n'$ that may be obtained from $G_n=G(A_n)$ by deleting
at most $\delta n$ vertices and their incident edges, and then adding
or deleting at most $\delta n$ edges.
\end{theorem}

As we shall show in \refSS{ss_stab}, using this result it is not hard
to deduce exponential tail bounds on the size of the giant component.

\begin{theorem}\label{th_cb}
Let $\ka$ be an irreducible kernel and $\eps>0$ a real number.
There is a $\gamma=\gamma(\ka,\eps)>0$ such that whenever
$(A_n)$ is sequence of non-negative symmetric $n$-by-$n$ matrices
with $\dcut(A_n,\ka)\to 0$, then setting $G_n=G(A_n)$ we have
\[
 \Pr\bb{ |C_1(G_n) -\rho(\ka)n | \ge \eps n} \le e^{-\gamma n}
\]
and
\[
 \Pr\bb{ C_2(G_n) \ge \eps n } \le e^{-\gamma n}
\]
for all large enough $n$.
\end{theorem}

For the very special case of $G(n,p)$, $p=c/n$, much stronger results
are known, establishing the correct dependence of
$\gamma$ on $\eps$ in the upper and lower bounds.
Indeed, such a `large deviation principle' 
for $C_1(G(n,c/n))$ was obtained by O'Connell~\cite{OC},
and Biskup, Chayes and Smith~\cite{BCS} proved a corresponding
result for the number of vertices in `large' components.
One might ask whether these results can be generalized to $G(A_n)$;
this is likely to be rather hard.
Indeed, it is not even clear whether they extend
to $G(A_n)$ with $A_n$ converging to a constant kernel $\ka$.

\begin{remark}\label{r_pto}
We have stated all our results for a deterministic sequence
$A_n$ with $\dcut(A_n,\ka)\to 0$. In applications, however, the matrices
$A_n$ are often random, and $G_n$ is defined by first
conditioning on $A_n$, and then taking the entries as giving
the conditional probabilities of the edges, which are conditionally
independent. The conclusions of Theorems~\ref{th1}--\ref{th_stab}
are all of the form that $G(A_n)$ has certain properties whp.
Having proved such a result assuming $\dcut(A_n,\ka)\to 0$,
the corresponding result with $A_n$ random and $\dcut(A_n,\ka)\pto 0$
follows immediately. One way of seeing this is to note
that a sequence $E_n$ of events holds whp if and only if
every subsequence has a subsubsequence holding whp.
If $\dcut(A_n,\ka)\pto 0$, then
given a subsequence (with deterministic indices) of
the random sequence $(A_n)$, one can find a subsubsequence
such that $\dcut(A_n,\ka)\to 0$ holds a.s., 
condition on the matrices
in this subsubsequence, and apply the result for the deterministic
case.
\end{remark}

\medskip
The rest of the paper is organized as follows. In the next few subsections
we discuss various applications and consequences of the results above.
In \refS{sec_proof} we prove Theorems~\ref{th1}--\ref{th_cb}: as the proofs
are somewhat lengthy we shall break this section into subsections.
Finally, in \refS{sec_hyp} we present extensions of our main results to the
{\em hyperkernels}
and corresponding random (hyper)graphs considered in~\cite{clustering}.

\subsection{Relationship to the sparse inhomogeneous model}\label{ss_rel}

In this subsection we shall prove a simple lemma which, together with 
\refT{th2}, implies Theorem 3.1 of~\cite{kernels}. This latter
result  states that (essentially) the conclusions of 
Theorems~\ref{th1} and \ref{th2} (with $c=1$)
hold when the random graph $G_n$ is an instance of
the general sparse inhomogeneous model $G^\vxs(n,\ka_n)$ of~\cite{kernels}.
Since the full definitions 
of~\cite{kernels} are rather cumbersome, for this subsection only we assume
a certain familiarity with the terminology of~\cite{kernels}.

We say that a kernel $\ka$ on $(\sss,\mu)$ is {\em of finite type}
if there is a finite partition $(S_1,\ldots,S_r)$ of $\sss$ into
measurable sets such that $\ka$ is constant on each of the sets $S_i\times S_j$.
A key strategy we used in~\cite{kernels} was to reduce results about the general
case to the finite-type case; we shall use
the same approach in this subsection.
In the rest of this paper we follow a different strategy, using
cut convergence to directly prove results about the general case.

The sparse inhomogeneous model $G^\vxs(n,\ka_n)$
is defined in terms of a {\em ground space} $\vxs=(\sss,\mu,\xss)$,
and a sequence $(\ka_n)$ of kernels on $(\sss,\mu)$.
Here $(\sss,\mu)$
is a probability space (satisfying some additional assumptions)
and each $\xs_n$ is a (deterministic or) random sequence of $n$ points
of $\sss$, satisfying certain technical assumptions.
The sequence $(\ka_n)$ is assumed to converge to a kernel $\ka$
in a certain sense, and must also satisfy
a certain `graphicality' assumption that involves the sequences $\xs_n$.
For the full technical details, which will not be relevant here, see~\cite{kernels}.

As noted in~\cite[Remark 8.8]{kernels},
in proving results about this model
one may always assume that the vertex types are deterministic. In this
case $G^\vxs(n,\ka_n)$ has the distribution of
$G(A_n)$, where $A_n$ is the matrix obtained
by sampling the kernel according to the vertex types: $A_n$ has
entries $a_{ij}=a_{ij}^{(n)}$ given by
$a_{ij}=\ka_n(x_i^{(n)},x_j^{(n)})\wedge n$ for $i\ne j$
and $a_{ii}=0$, where $x\wedge y=\min\{x,y\}$.
We refer the reader to \cite{kernels} for the formal definition of $G^\vxs(n,\ka_n)$,
and in particular for the precise definitions of a (generalized) vertex space
and a graphical (sequence of) kernel(s).

The next lemma shows that the matrices $A_n$ associated to $G^\vxs(n,\ka_n)$
do converge in probability to the limit kernel $\ka$ in the cut metric.
Although our main interest is in the cut distance, we in fact obtain a result
for the $L^1$ norm, modulo rearrangements. Given two 
kernels $\ka$, $\ka'$ on the standard ground space,
let
\begin{equation}\label{d1def}
 \delta_1(\ka,\ka')=\inf_{\ka''\sim \ka'} \on{\ka-\ka''},
\end{equation}
in analogy with \eqref{cutdef}. More generally, for two kernels
on arbitrary (not necessarily equal) probability spaces, we
may define $\delta_1(\ka,\ka')$ as a certain infimum over couplings
of these probability spaces; we omit the details.

\begin{lemma}\label{l_k}
Let $\vxs=(\sss,\mu,\xss)$ be a vertex space,
and let $(\ka_n)$ be a sequence of kernels that
is graphical on $\vxs$ with limit $\ka$. 
Let $A_n$ be the matrix with entries
$a_{ij}=\ka_n(x_i^{(n)},x_j^{(n)})\wedge n$
for $i\ne j$ and $a_{ii}=0$.
Then $\delta_1(\ka_{A_n},\ka)\pto 0$ 
and $\dcut(A_n,\ka)=\dcut(\ka_{A_n},\ka)\pto 0$.
\end{lemma}

\begin{proof}
Since $\cn{\ka'}\le \on{\ka'}$ for any $\ka'$, we have
$\dcut(\ka_1,\ka_2)\le \delta_1(\ka_1,\ka_2)$ for any two kernels,
so it suffices to prove the first statement.

Conditioning on the vertex types, we may and shall assume
that the vertex types are deterministic.
For convenience we assume that $\sss$ is the standard ground space
$\oi$. (The general case requires couplings of $\ka$ and $A_n$, but is
otherwise the same.)

Suppose first that $\ka$ is regular finitary; roughly speaking,
this means that $\ka$ is of finite type. (More precisely,
$\ka$ must be of finite type and must satisfy an additional technical condition;
see~\cite{kernels}.)
Suppose also that $\ka_n=\ka$ for every $n$.
In this case the result is essentially trivial:
we may assume that
there is a partition of $\sss$ into sets $S_1,\ldots,S_k$ such that
$\ka$ is constant on each 
set $S_r\times S_s$. The definition of a vertex space ensures that for each $r$
there are $\mu(S_r)n+o(n)$ vertices $i$ such that $x_i\in S_r$. 
Rearranging (or coupling) appropriately, we may assume that each $S_r$
is an interval $I_r\subseteq\sss=\oi$.
We may then
order the vertices so that for all but $o(n)$ vertices $i$ the interval
$(i-1/n,i/n]$ lies entirely inside the interval $I_r$ containing $x_i$.
After doing so, $\ka$ and $\ka_{A_n}$ differ on a set of measure
$o(1)$. 
Since
both are bounded by $\sup\ka<\infty$, it follows that $\ka_{A_n}\to \ka$ in $L^1$
and hence in $\dcut$.

To treat the general case, we approximate by finite-type kernels, as so often in~\cite{kernels}.
Indeed, by Lemma 7.3 of~\cite{kernels} there is a sequence of regular finitary kernels
$\ka_m^-$ such that
$\ka_m^-\le \ka_n$ for all $n\ge m$ and $\ka_m^-(x,y)\upto \ka(x,y)$
for a.e.\ $(x,y)\in \sss^2$.
By monotone convergence, we have $\int \ka_m^-\to \int \ka$ as $m\to \infty$.
Fix $\eps>0$. Then there is some $m$ such that $\ka^-=\ka_m^-$ satisfies
$\ka^-\le \ka$ and $\int (\ka-\ka^-)\le \eps$.

Let $A_n^-$ be the matrix with entries
$a_{ij}^-=\ka^-(x_i^{(n)},x_j^{(n)})\wedge n$, $i\ne j$,
and $a_{ii}^-=0$.
Considering from now on only $n\ge m$, we then have
$a_{ij}^-\le a_{ij}$ and thus
$\ka_{A_n^-}\le \ka_{A_n}$ pointwise.
After conditioning on the vertex types, the expected number of edges in $G^\vxs(n,\ka_n)$
is exactly
\[
 \frac{1}{2} \sum_i \sum_{j\ne i} \frac{a_{ij}}{n} = 
  \frac{1}{2} \sum_i \sum_j \frac{a_{ij}}{n} = \frac{n}{2} \int \ka_{A_n},
\]
using $a_{ii}=0$ for the first equality.
Thus, by Lemma 8.7 of~\cite{kernels}, $\int \ka_{A_n}\to \int\ka$.
Similarly (since a finite-type kernel is always graphical),
$\int \ka_{A_n^-}\to \int\ka^-$.
Hence,
\[
 \on{\ka_{A_n}-\ka_{A_n^-}} = \int (\ka_{A_n}-\ka_{A_n^-})
 \to \int (\ka-\ka^-) \le \eps.
\]
By the finite-type case above, we have $\delta_1(\ka_{A_n^-},\ka^-)\to 0$.
Since $\on{\ka-\ka^-}\le \eps$ it follows
that $\limsup \delta_1(\ka_{A_n},\ka)\le 2\eps$.
Recalling that $\eps>0$ was arbitrary, the result follows.
\end{proof}

Recall that Theorem 3.1 of~\cite{kernels} states (essentially)
that the random graphs 
$G_n=G^\vxs(n,\ka_n)$ satisfy the conclusions of
Theorems~\ref{th1} and \ref{th2}. 
Using \refL{l_k}, by Remark~\ref{r_pto}
the vertex space case of this result follows
immediately from Theorems \ref{th1} and \ref{th2}. 
As noted in \cite[Section 8.1]{kernels},
the apparent extra generality of generalized vertex spaces makes no essential
difference, so Theorem 3.1 of~\cite{kernels} then follows.
In other words, we have shown that Theorem 3.1 of~\cite{kernels} may
be deduced from our present Theorems \ref{th1} and \ref{th2},
using various results from~\cite{kernels} mentioned above.
Let us remark that in practice,
the conditions of Theorem 3.1 of~\cite{kernels} will often
be easier to verify than those of Theorems \ref{th1} and \ref{th2}.

\subsection{Further applications}

As noted in~\cite{clustering}, the definitions in~\cite{kernels} exclude
one simple case to which the results clearly extend, namely the case
of an arbitrary integrable kernel $\ka$, and \iid\ vertex types: 
given a kernel $\ka$, one may define the random graph $G(n,\ka)=G_{1/n}(n,\ka)$
on $[n]$ by taking $x_1,\ldots,x_n$ to be independent and uniformly distributed
on $[0,1]$, and given these `vertex types', joining each pair $\{i,j\}$ of vertices
with probability $\min\{\ka(x_i,x_j)/n,1\}$, independently of all other pairs.
With $\ka$ bounded, a corresponding dense random graph was studied by
Lov\'asz and Szegedy~\cite{LSz1}.

Our next lemma shows that Theorems \ref{th1}--\ref{th_stab}
apply (unsurprisingly) to the
graphs $G(n,\ka)$, since the (random) matrices of edge probabilities associated to $G(n,\ka)$
converge to $\ka$ in probability in $\dcut$.  

\begin{lemma}\label{l_iid}
Let $\ka$ be a kernel. For $n\ge 1$ let $x_1,\ldots,x_n$ be \iid\ uniform points
from $\sss$, and let $A_n$ be the $n$-by-$n$ matrix with entries $a_{ij}=\ka(x_i,x_j)$
for $i\ne j$, and $a_{ii}=0$.
Then $\delta_1(A_n,\ka)\pto 0$ and $\dcut(A_n,\ka)\pto 0$.
\end{lemma}
\begin{proof}
As before, we have $\dcut\le\delta_1$, so it suffices to prove the first statement.
Fix $\eps>0$. By standard results there is a finite-type kernel $\ka'$
such that $\on{\ka-\ka'}\le \eps^2$. 
Indeed, this follows by the construction of the product measure, since the
rectangular sets $A\times B$ generate an algebra $\cF_0$ that
generates the product $\sigma$-field, and it is easily seen that finite
linear combinations of indicator functions of sets in $\cF_0$ are
dense in $L^1(\sss^2)$.

Let $A_n'$ be the matrix with entries $a_{ij}'=\ka'(x_i,x_j)$, $i\ne j$, and $a_{ii}'=0$.
Then
\[
 \E \on{\ka_{A_n}-\ka_{A_n'}} = \frac{n(n-1)}{n^2} \on{\ka-\ka'} \le \eps^2,
\]
so with probability at least $1-\eps$ we have
\begin{equation}\label{2eps}
 \delta_1({A_n},{A_n'}) =
 \delta_1(\ka_{A_n},\ka_{A_n'}) \le \on{\ka_{A_n}-\ka_{A_n'}}\le \eps.
\end{equation}
Since $\ka'$ is of finite type, it is essentially trivial that
$\delta_1(A_n',\ka')\pto 0$ as $n\to\infty$; the argument is similar to
one in the previous subsection, so we omit the details.
Since $\delta_1(\ka,\ka')\le \on{\ka-\ka'} \le \eps^2$,
\[
 \delta_1(A_n,\ka) \le \delta_1(A_n,A_n')+\delta_1(A_n',\ka')+\delta_1(\ka',\ka),
\]
and $\eps>0$ was arbitrary, it follows that $\delta_1(A_n,\ka)\pto 0$, as claimed.
\end{proof}

So far we have shown that the results in Subsection~\ref{ss_res}
imply many existing
results about the giant component in various sparse random graphs. We
now turn to a new application, giving an example that we believe
is not covered by known results.

Let $p=p(n)$ be some normalizing function, with $0<p\le 1$ and $p(n)\to 0$.
Let $G_n$ be a sequence of graphs in which $G_n$ has $n$ vertices
and $\Theta(pn^2)$ edges, and let $\ka$ be a kernel.
Following the terminology of~\cite{BRsparse,BRsparse2}, we say
that $\dcut(G_n,\ka)\to 0$ if $\dcut(A_n,\ka)\to 0$, where $A_n$
is $1/p$ times the adjacency matrix of $G_n$.
A sequence $(G_n)$ satisfying this condition may be thought of as a sequence
of inhomogeneous sparse quasi-random graphs. For graphs which are dense
and homogeneous,
there are many equivalent definitions of quasi-randomness,
or pseudo-randomness; see Thomason~\cite{Tho,Tho2} or
Chung, Graham and Wilson~\cite{CGW89}, for example.
In the sparse case these notions are no longer equivalent,
as discussed by Chung and Graham~\cite{CG} in the homogeneous case,
and Bollob\'as and Riordan~\cite{BRsparse} in general; when $\ka$
is constant, normalizing so that $\ka=1$,
we have $\dcut(G_n,\ka)\to 0$ if and only if
\begin{equation}
 \sup_{V\subset V(G_n)} \bigl| e(G_n[V]) - p|V|^2/2 \bigr| =o(pn^2);
\end{equation}
this condition is called DISC in~\cite{CG}.
Other, stronger conditions have also been considered,
in particular by Thomason~\cite{Tho,Tho2}.
Our next result establishes the threshold for percolation on an arbitrary
sequence of inhomogeneous sparse quasi-random graphs.

\begin{theorem}\label{th_qrs}
Let $c>0$ be a constant, let $p=p(n)$ be any function with $c/n \le p(n)\le 1$, let $\ka$
be an irreducible kernel on $[0,1]^2$, and let $(G_n)$ be a sequence of graphs
with $|G_n|=n$ and $\dcut(G_n,\ka)\to 0$.
Writing $G_n'$ for the random subgraph of $G_n$ obtained by selecting
each edge independently with probability $c/(pn)$,
we have $C_1(G_n')/n\pto \rho(c\ka)$. In particular, the threshold
value of $c$ above which a giant component appears in $G_n'$ is given by $1/\norm{T_\ka}$.
\end{theorem}
\begin{proof}
As above, let $A_n$ be $1/p$ times the adjacency matrix of $G_n$.
Then, by assumption, $\dcut(A_n,\ka)\to 0$, so $\dcut(c A_n,c\ka)\to 0$.
The random subgraph $G_n'$ is exactly $G(c A_n)$, so the result
follows from \refT{th1}.
\end{proof}

As noted in~\cite{BRsparse}, one way to construct inhomogeneous
sparse quasi-random graphs is to consider appropriate {\em random} graphs, but this
is not so interesting in the present context: the random subgraphs of such graphs
end up being the graphs $G(n,\ka)$ considered at the start of the subsection.
A more interesting application of \refT{th_qrs} is to deterministic
quasi-random graphs. In the homogeneous case, where $\ka=1$ is constant,
many such sequences
are known. One example is given by the `polarity graphs'
of Erd\H os and R\'enyi~\cite{ERpolarity},
defined (for suitable $n$)
by taking as vertices the points of the projective
plane over $GF(q)$, $q$ a prime power, and joining $x=(x_0,x_1,x_2)$
and $y=(y_0,y_1,y_2)$ 
if and only if $x_0y_0+x_1y_1+x_2y_2=0$ in $GF(q)$. Here
$n=q^2+q+1$ and $p=(q+1)/n=\Theta(n^{-1/2})$. 
Other examples are the coset graphs of Chung~\cite{Chung_coset}
and the Ramanujan graphs of Lubotzky, Phillips and Sarnak~\cite{LPS}. 
In all these examples the limiting kernel is constant, so \refT{th_qrs}
says that on any of these graphs, the threshold for percolation
is when the average degree of the random subgraph is equal to $1$.

Note that in the examples above, the matrices $(A_n)$ to which \refT{th1}
or \refT{th2} is applied are very far from satisfying the uniform
boundedness condition assumed in
Bollob\'as, Borgs, Chayes and Riordan~\cite{QRperc}. Indeed,
each $A_n$ has all entries either $0$ or $1/p$, where $p=p(n)\to 0$.
This also implies that the corresponding kernels $\ka_{A_n}$,
which do converge to $\ka=1$ in the cut norm, do not converge
in various natural stronger senses, such as pointwise or in $L^1$.

In general, it is very hard to compute the cut distance between two
kernels. Indeed, if $A_1$ and $A_2$ are the adjacency matrices of
two graphs, then the general problem of computing $\dcut(\ka_{A_1},\ka_{A_2})$
includes as a special case
deciding whether $G_1$ and $G_2$ are isomorphic.
Thus applications of Theorems~\ref{th1} and~\ref{th2} are likely to involve special cases
where cut convergence is guaranteed for some simple reason,
such as the example in the previous subsection.

\subsection{Consequences for branching processes}

\refT{th1} has an interesting consequence purely concerning branching
processes. Recall that if $\ka$ is a kernel, then $\rho(\ka)$ denotes
the survival probability of the multi-type Poisson Galton--Watson process
$\bpk$.

\begin{theorem}\label{th_br}
Let $\ka_m$, $m\ge 1$, and $\ka$ be kernels with
$\dcut(\ka_m,\ka)\to 0$
as \mtoo.
Then $\rho(\ka_m)\to \rho(\ka)$.
\end{theorem}
\begin{proof}
Let us first note that the result is not really a statement about
the cut metric $\dcut$, but rather about the cut norm $\cn{\cdot}$.
Indeed, by definition of $\dcut$
there are rearrangements $\ka_m'$ of $\ka_m$ with $\cn{\ka_m'-\ka}\le\dcut(\ka_m,\ka)+1/m$,
say, and hence $\cn{\ka_m'-\ka}\to 0$. Since $\rho(\ka_m')=\rho(\ka_m)$, in proving
the result we may assume if we like that $\cn{\ka_m-\ka}\to 0$.

We shall prove the result in three steps.

Step 1: suppose that all $\ka_m$ are irreducible; 
this case is the heart of the proof.
For each $m$ we may find a sequence $A_n^{(m)}$ 
of symmetric $n$-by-$n$ matrices
with $\dcut(A_n^{(m)},\ka_m)\to 0$ as \ntoo. 
Indeed, this is an immediate
consequence of Lemma~\ref{l_iid}.
By \refT{th1}, if $n$ is large enough, then
\begin{equation}\label{nn}
 \Pr\Bigpar{\bigabs{C_1(G(A_n^{(m)}))/n - \rho(\ka_m)} \ge 1/m} \le 1/m^2,
\end{equation}
say. Pick $n(m)$ such that \eqref{nn} holds and 
$\dcut(A_{n(m)}^{(m)},\ka_m)\le 1/m$,
and let $A_m=A_{n(m)}^{(m)}$.
By \eqref{nn}, with probability $1$ we have
\begin{equation}\label{cccl}
 \lrabs{\frac{C_1(G(A_m))}{|G(A_m)|} - \rho(\ka_m) } \to 0.
\end{equation}
Now $\dcut(A_m,\ka_m)\le 1/m$ by our choice of $n(m)$,
while $\dcut(\ka_m,\ka)\to 0$, so $\dcut(A_m,\ka)\to 0$.
Applying \refT{th1} again,
we have $C_1(G(A_m))/|G(A_m)|\le \rho(\ka)+\op(1)$. Together
with  \eqref{cccl} this implies that
\begin{equation}\label{ls}
 \limsup \rho(\ka_m) \le \rho(\ka).
\end{equation}
If $\ka$ is irreducible, then we have $C_1(G(A_m))/|G(A_m)|\pto \rho(\ka)$,
so $\rho(\ka_m)\to \rho(\ka)$, as required. We shall return
to the lower 
bound in the case that $\ka$ is reducible later.

Step 2: we now consider the general case, where some of $\ka$ and the $\ka_m$ may be reducible.
By Theorem~6.4\nobreak(i) of \cite{kernels}, given a kernel $\ka'$ and a sequence
$\ka_n'$ tending pointwise down to $\ka'$, we have $\rho(\ka_n')\to\rho(\ka')$.
Applying this with $\ka'=\ka_m$ and $\ka_n'=\ka_m+1/n$, say,
we see that for each $m$ there is an $\eps_m<1/m$ such that
$|\rho(\ka_m')-\rho(\ka_m)|\le 1/m$, where $\ka_m'=\ka_m+\eps_m$.
Now $\ka_m'$ is irreducible, and $\cn{\ka_m'-\ka_m}\le 1/m\to 0$,
so $\dcut(\ka_m',\ka)\to 0$, and the results
of Step 1 apply. In particular, the upper bound \eqref{ls} holds,
and if $\ka$ is irreducible, then $\rho(\ka_m)\to \rho(\ka)$, as required.

Step 3: in the case where $\ka$ is reducible,
it remains to prove the lower bound corresponding to \eqref{ls}.
For this we decompose $\ka$ into irreducible kernels
as in~\cite{kernels}.
As shown there (in Lemma 5.17),
given any $\ka$ there is a finite or countable partition
$(S_i)_{i=0}^N$, $N\le\infty$, of $\sss$ into measurable sets
such that $\ka=\sum_{i\ge 1} \ka^{(i)}$ holds a.e., where each $\ka^{(i)}$ 
is zero off $\sss_i\times \sss_i$ and irreducible when restricted
to $\sss_i\times \sss_i$.
Fix $\eps>0$.
Since $\rho(\ka)=\sum \rho(\ka^{(i)})$, there is some $k<\infty$ such that
$\sum_{i=1}^k \rho(\ka^{(i)})\ge \rho(\ka)-\eps$.
Define $\ka_m^{(i)}$ to be the kernel that is equal to $\ka_m$ on $\sss_i\times\sss_i$
and zero off this set, and let $\ka_m'=\sum_{i=1}^k \ka_m^{(i)}$.
Then $\ka_m\ge \ka_m'$,
so $\rho(\ka_m)\ge \rho(\ka_m')=\sum_{i=1}^k \rho(\ka_m^{(i)})$.
Since $\cn{\ka_m-\ka}\ge  \cn{\ka_m^{(i)}-\ka^{(i)}}$ for each $i$, we have
$\cn{\ka_m^{(i)}-\ka^{(i)}}\to 0$ for each $i$. Since $\ka^{(i)}$
is irreducible, by the result of Step 2
we have $\rho(\ka_m^{(i)})\to \rho(\ka^{(i)})$.
Summing over $i$ from $1$ to $k$ it follows that
\[
 \liminf_{m\to\infty} \rho(\ka_m) \ge \sum_{i=1}^k \rho(\ka^{(i)}) \ge \rho(\ka)-\eps.
\]
Since $\eps>0$ was arbitrary we thus have
$\liminf_{m\to\infty} \rho(\ka_m) \ge \rho(\ka)$.
Together with \eqref{ls}, this completes the proof.
\end{proof}

Note that \refT{th_br} is a purely analytic statement about
branching processes and the cut metric (or cut norm -- rearrangements
change nothing here). However, the only proof we know is that above,
which goes via graphs! Corresponding results with much stronger
assumptions (monotone convergence, either upwards or downwards)
were proved in~\cite{kernels}; these weaker results were all that
was needed there.

We close this section by giving a direct
proof of a weaker form of \refT{th_br}, assuming $L^1$ convergence.
As above, rearrangement is irrelevant, so it makes no difference
whether we suppose that $\delta_1(\ka_n,\ka)\to 0$ or $\on{\ka_n-\ka}\to 0$.

\begin{theorem}\label{th_L1}
Let $\ka_n$, $n\ge 1$, and $\ka$ be kernels on a probability space $(\sss,\mu)$, with
$\on{\ka_n-\ka}\to 0$ as \ntoo.
Then $\rho(\ka_n)\to \rho(\ka)$.
\end{theorem}
The proof will be based on weak-$*$ convergence.
Let $f_n$, $n\ge 1$, and $f$ be functions in $L^\infty(\sss,\mu)$.
The definition of the weak-$*$ topology on $L^\infty(\sss,\mu)$
is that $f_n\wst f$ if and only if
\begin{equation}\label{wsdef}
 \int g(x)f_n(x) \dd\mu(x) \to \int g(x) f(x)\dd\mu(x)
 \hbox{ for every }g\in L^1(\sss,\mu).
\end{equation}

\begin{lemma}\label{kw}
Suppose that $\ka\in L^1(\sss\times\sss)$ and
$f_n\in L^\infty(\sss,\mu)$
with $f_n\wst 0$.
Let $h_n=T_\ka f_n$, so
$h_n(x)=\int \ka(x,y) f_n(y)\dd\mu(y)$.
Then $h_n\to 0$ in $L^1(\sss,\mu)$.
\end{lemma}
\begin{proof}
Note first that by the uniform boundedness principle we have
$C=\sup \sn{f_n}<\infty$. (In fact, in the application, each $f_n$ is bounded by $1$.)

Let $\eps>0$. As in the proof of \refL{l_iid},
there is a finite-type kernel $\ka'$
such that $\on{\ka-\ka'}<\eps$.
We may express $\ka'$ as
$\ka'(x,y)=\sum_{i=1}^N \varphi_i(x)\psi_i(y)$ for $\varphi_i$,
$\psi_i\in L^1$. 
(In fact, we may take each $\varphi_i$ or $\psi_i$ to be a constant
times a characteristic function.)
Now
\begin{multline*}
 \on{h_n} = \onb{\int \ka(x,y)f_n(y)\dd\mu(y)}  \\
 \le \int\left| (\ka(x,y)-\ka'(x,y))f_n(y) \right|\dd\mu(x)\dd\mu(y)
 + \sum_{i=1}^N \onb{ \int \varphi_i(x)\psi_i(y)f_n(y)\dd\mu(y) }.
\end{multline*}
The first term above is at most $\on{\ka-\ka'}\sn{f_n}\le \eps C$.
The second term is exactly
\[
 \sum_{i=1}^N \on{\varphi_i}\lrabs{\int \psi_i(y)f_n(y)\dd\mu(y)}.
\]
Each integral tends to zero by the definition \eqref{wsdef} of the
weak-$*$ topology, so it follows that $\limsup \on{h_n}\le \eps C$.
Since $\eps>0$ was arbitrary, the result follows.
\end{proof}

With this preparation behind us, we turn to the proof of \refT{th_L1}.
\begin{proof}[Proof of Theorem~\ref{th_L1}]
We may assume without loss of generality that the $\sigma$-field $\cF$
on $\sss$ where $\mu$ is defined is countably generated, and thus
$L^1(\sss,\mu)$ is separable.  
One way to
see this is to note that otherwise we can replace $\cF$ by a countably
generated sub-$\sigma$-field $\cF_0$ such that each $\ka_n$ is
$\cF_0\times\cF_0$-measurable; alternatively,
by the results of~\cite{SJ210} we may assume without loss of
generality 
that $\sss=[0,1]$, with $\mu$ Lebesgue measure.

Suppose for simplicity that $\ka$ is irreducible;
arguing as in the proof of Theorem~\ref{th_br}, it is not hard to reduce 
the general case to this case.

Suppose for a contradiction that $\on{\ka_n-\ka}\to 0$ but $\rho(\ka_n)\not\to\rho(\ka)$.
Passing to a subsequence, we may assume that $|\rho(\ka_n)-\rho(\ka)|$ is bounded
away from zero. To obtain a contradiction it then
suffices to show that for some subsequence $(\ka_{n_i})$
of $(\ka_n)$ we have $\rho(\ka_{n_i})\to\rho(\ka)$.

Let $\rho_n(x)=\rho_{\ka_n}(x)$ be the survival probability of the branching
process $\bp_{\ka_n}(x)$, started with a single particle of type $x$.
As shown in~\cite{kernels}, the function $\rho_n$ satisfies
\begin{equation}\label{rneq}
 \rho_n=1-\exp(-T_{\ka_n} \rho_n).
\end{equation}
It is well known that
the unit ball of $L^\infty(\sss,\mu)$ 
is sequentially compact in the
weak-$*$ topology when $L^1(\sss,\mu)$ is separable. (The unit ball of $L^\infty$ is always
compact, but not necessarily sequentially compact otherwise.)
For the special case $\sss=[0,1]$, let $(f_n)$ be a sequence
in the unit ball of $L^\infty([0,1])$. This sequence
has a subsequence $(f_{n_k})$ such that $\int_I f_{n_k}$ converges 
for each of the countably many intervals $I$ with rational endpoints.
Since the $f_{n_k}$ are uniformly bounded, this is enough to ensure
weak-$*$ convergence.

Since $\sn{\rho_n}\le 1$ for every $n$, by sequential compactness there
is some $\rhos\in L^\infty(\sss,\mu)$ and some subsequence of $(\ka_n)$
along which $\rho_n\wst\rhos$. From now on we restrict our attention to such a
subsequence.

Now
\[
 \on{T_{\ka_n}\rho_n - T_\ka \rho_n} \le \on{\ka_n-\ka}\sn{\rho_n} \le \on{\ka_n-\ka} \to 0.
\]
Also, by \refL{kw}, $\on{T_\ka \rho_n -T_\ka \rhos}\to 0$.
Hence $T_{\ka_n}\rho_n \to T_\ka\rhos$ in $L^1$.
Passing to a subsequence, we may assume that $T_{\ka_n}\rho_n \to T_\ka\rhos$ a.e.
But then, using \eqref{rneq},
\[
 \rho_n = 1- e^{-T_{\ka_n}\rho_n} \to 1-e^{-T_\ka \rhos} \hbox{ a.e.}
\]
From \eqref{wsdef} and dominated convergence, it follows that
\[
 \rho_n \wst 1-e^{-T_\ka \rhos}.
\]
Since $\rho_n\wst \rhos$, it follows that $\rhos = 1-e^{-T_\ka \rhos}$ a.e.

Let $\rho(x)$ denote the survival probability of $\bp_{\ka}(x)$.
Since $\ka$ is irreducible, by \cite[Theorem 6.2]{kernels},
either $\rhos=\rho$ a.e.\ or $\rhos=0$ a.e.
In the first case, 
\[
 \rho(\ka_n) = \int \rho_n(x)\dd\mu(x) \to \int \rhos(x)\dd\mu(x) =\rho(\ka),
\]
as desired.
In the second case, we have $\rho(\ka_n)\to 0$ similarly.

All that remains is to rule out the possibility that $\rho(\ka_n)\to 0<\rho(\ka)$.
This is not hard using the results in~\cite{kernels}.
For $M>0$, let $\ka^M$ denote the pointwise minimum of $\ka$ and $M$,
and define $\ka_n^M$ similarly. Suppose that $\rho(\ka)>0$. Then
$\norm{\Tk}>1$. As shown in the proof of
\cite[Lemma 5.16]{kernels},
we have $\norm{T_{\ka^M}}\upto \norm{\Tk}$
as $M\to\infty$, so there is some $M$ with $c=\norm{T_{\ka^M}}>1$. Fix such an $M$.
Since
\begin{equation}\label{contr}
 \on{\ka_n^M -\ka^M} \le \on{\ka_n-\ka} \to 0,
\end{equation}
and the kernels $\ka_n^M$ and $\ka^M$ are uniformly bounded, we have
$\norm{T_{\ka_n^M}}\to \norm{T_{\ka^M}}=c>1$. In particular, for all
large enough $n$ we have $\norm{T_{\ka_n^M}} > (c+1)/2 >1$.
Finally, it follows from \cite[Remark 5.14]{kernels} that we have
\[
 \rho(\ka_n^M) \ge \frac{\norm{T_{\ka_n^M}} -1}{\sup \ka_n^M} \ge \frac{(c-1)/2}{M} >0.
\]
Since $\rho(\ka_n)\ge \rho(\ka_n^M)$ it follows that $\rho(\ka_n)\not\to 0$, and the proof
is complete.
\end{proof}

If we assume cut convergence instead of $L^1$ convergence, then using the fact that
\[
 \onb{\int \ka(x,y) f(y)\dd\mu(y)} \le \cn{\ka}\sn{f}
\]
in place of the corresponding observation for the $L^1$ norm, the first part of the
proof above goes through unchanged, showing that $\rhos\to \rho$
a.e.\ or $\rhos\to 0$. 
Unfortunately, we do not know how to exclude the possibility that $\rho(\ka_n)\to 0<\rho(\ka)$,
except by appealing to \refT{th1}, i.e., working with graphs.
The problem is that the relation equivalent to \eqref{contr} for the cut norm rather
than the $L^1$ norm does not hold in general.
Of course, given that \refT{th_br} is true, it is almost guaranteed that 
it has a direct analytic proof.

As discussed in~\cite[Section 2]{BRsparse}, until recently there was
another
example of an analytic fact about kernels whose only known proof involved
graphs (and the cut metric), namely that two bounded kernels may be coupled
to agree a.e.\ if and only if their `graphical moments' (or subgraph counts)
are equal. This follows from the results of Borgs, Chayes, Lov\'asz, S\'os
and Vesztergombi~\cite{BCLSV:1} concerning metrics for graphs (see~\cite{BRsparse}).
However, by now there are analytic proofs:
Janson and Diaconis~\cite{SJ209} showed that it also follows from
results of Hoover and Kallenberg on exchangeable arrays.
A direct (and far from simple) proof has recently been given by
Borgs, Chayes and Lov\'asz~\cite{BCL:unique}.

\section{Proofs of Theorems \ref{th1}--\ref{th_cb}}\label{sec_proof}

In this section we shall prove our main results;
the strategy of the proof of \refT{th1} is as follows. First,
in Subsection~\ref{ss_elim}, we shall
show that if each $\ka_n$ is an $n$-by-$n$ kernel and $\dcut(\ka_n,\ka)\to 0$,
then almost all of the weight of $\ka_n$ comes from values
that are $o(n)$. This will allow us to assume that all edge probabilities
in $G(A_n)$ are $o(1)$.
It then follows that the expected number of small tree components
in $G(A_n)$ is close to what it `should be', i.e., $n$ times a certain
function of the kernel $\ka_{A_n}$.
In Subsection~\ref{ss_ti} we show
that this function is continuous with respect to the cut metric.
This then tells us that we have almost the `right' number of vertices
in small components; the details are given in Subsection~\ref{ss_s}.
Finally, in Subsection~\ref{ss_conn} we complete
the proof of \refT{th1} by showing that 
in the irreducible case, almost
all vertices in large components are in a single component, using a method
from Bollob\'as, Borgs, Chayes and Riordan~\cite{QRperc}.
In Subsection~\ref{ss_red} we treat the reducible case, proving \refT{th2}.
Finally, in Subsection~\ref{ss_stab} we prove our stability and
concentration results, Theorems~\ref{th_stab} and~\ref{th_cb}.

For convenience, in this section we assume, as we may, that
all kernels are on $[0,1]$, unless explicitly stated otherwise. 

\subsection{Eliminating large edge weights}\label{ss_elim}

In Theorem 2.1 of~\cite{BRsparse2}
it was shown that if $(G_n)$ is a sequence
of graphs in which $G_n$ has $n$ vertices and $O(n)$ edges,
$A_n$ is the adjacency matrix of $G_n$,
$\ka$ is a kernel and $\dcut(nA_n,\ka)\to 0$, then $\ka=0$ a.e.\ and $e(G_n)=o(n)$.
A simple modification of the proof gives the following lemma. Recall
that a matrix denoted $A_n$ is assumed to be $n$-by-$n$.

\begin{lemma}\label{l_on}
Suppose that $\ka$ is a kernel and $(A_n)$ a sequence of non-negative matrices
such that $\dcut(A_n,\ka)\to 0$.
Then there is some function $M(n)$ with $M(n)=o(n)$ such that only $o(n)$ entries
of $A_n$ exceed $M(n)$, and the sum of these entries is $o(n^2)$.
\end{lemma}

A consequence of this is that if $A_n'$ is obtained from $A_n$ by taking the pointwise
minimum with $M(n)$, then $\dcut(A_n',\ka)\to 0$.

\begin{proof}
Although the details are almost exactly the same as in~\cite{BRsparse2}, we spell
them out. We write $\ka_n$ for $\ka_{A_n}$.

Since $\dcut(\ka_n,\ka)\to 0$, we may choose rearrangements
$\karn$ of $\ka$ such that
\begin{equation}\label{conv}
 \cn{\ka_n-\karn}\to 0.
\end{equation}

It suffices to show that for any $c>0$, the sum of the entries of $A_n$ exceeding
$cn$ is at most $c^2n^2$ for $n$ large enough. This implies that there are
at most $cn$ such entries, and the result then follows by
letting $c$ tend to $0$.

Suppose for a contradiction that there is some $c>0$ 
such that, for infinitely many $n$,
the sum of the entries of $A_n$ exceeding $cn$ is at least $c^2n^2$;
from now on we fix such a $c$ and restrict our attention to the corresponding
values of $n$.
Let $G_n$ be the graph whose edges correspond to those entries of $A_n$ which exceed
$cn$. Let $M_n$ be a largest matching in $G_n$.

Suppose first that $|V(M_n)|/n\to 0$. Let $S_n$ be the subset of $[0,1]$
corresponding to the vertex set of $M_n$, so $\mu(S_n)=|V(M_n)|/n\to 0$.
Every edge of weight at least $cn$ meets a vertex of $M_n$, so
\[
 \int_{S_n\times [0,1]} \ka_n =\frac{1}{n^2} \sum_{v\in V(M_n)}\sum_w a_{vw}
  \ge \frac{1}{2n^2} (cn)^2 = c^2/2,
\]
where the factor 2 accounts for the double counting of edges within $V(M_n)$.

From \eqref{conv}, writing $S_n'$ for $\rn(S_n)$, we have
\[
 \int_{S_n'\times [0,1]} \ka = \int_{S_n\times [0,1]} \karn
 \ge \int_{S_n\times[0,1]} \ka_n-o(1) \ge c^2/2-o(1),
\]
so $\int_{S_n'\times [0,1]}\ka \not\to 0$. Since $\mu(S_n'\times [0,1])=\mu(S_n')=\mu(S_n)\to 0$,
this contradicts integrability of $\ka$.

Passing to a subsequence, we may thus assume that for some $a>0$, every 
maximal matching $M_n$ meets at least $an$ vertices.

Since $\ka$ is integrable, we have $\int \ka 1_{\{\ka>C\}}\to 0$ as $C\to\infty$,
where $1_{\{\ka>C\}}: [0,1]^2\to \{0,1\}$
is the indicator of the event that $\ka(x,y)>C$. In particular,
there is a $C<\infty$ with $\int \ka 1_{\{\ka>C\}}\le ac/4$.
Fix an $n$ with $n>4C/(ac)$, noting
that if $S\subset [0,1]^2$ satisfies $\mu(S)\le 1/n$, then
\begin{equation}\label{smsm}
 \int_S \ka \le C\mu(S) + \int \ka 1_{\{\ka>C\}} \le C/n+ac/4 \le ac/2.
\end{equation}
Choosing $n$ large enough, we may assume from \eqref{conv} that
there is a $\ka'=\karn \sim\ka$ with
\begin{equation}\label{cnc}
 \cn{\ka_n-\ka'}\le ac/25.
\end{equation}

Given subsets $U$ and $\VW$ of $[n]$, let
\[
 A_n(U,\VW) = \sum_{u\in U}\sum_{\vw\in \VW} a_{u\vw}.
\]

Let $M_n=\{u_1\vw_1,\ldots,u_r\vw_r\}$ be a matching in $G_n$ with $r\ge an$,
and set $U=\{u_i\}$ and $\VW=\{\vw_i\}$.
Identifying subsets of $[n]$ with the corresponding unions
of intervals of length $1/n$,
from \eqref{cnc} we have
\[
 \left| \int_{U\times \VW}\ka'  - \frac{A_n(U,\VW)}{n^2} \right| \le ac/25.
\]
Let $U'$ be a random subset of $U$ obtained by selecting each vertex
independently with probability $1/2$, and let $\VW'$ be the complementary
subset of $\VW$, defined by $\VW'=\{\vw_i: u_i\notin U_i\}$.
The edges of our matching $M_n$ never appear as edges from $U'$ to $\VW'$.
On the other hand, any other edge $u_i\vw_j$, $i\ne j$,
from $U$ to $\VW$ has probability $1/4$ of appearing. Hence,
\[
 \E\bb{A_n(U',\VW')} = \frac{A_n(U,\VW)}{4} - \frac{1}{4}\sum_{i}A_{u_i\vw_i}.
\]
Similarly, writing $S\subset [0,1]^2$ for the union of the $r$
$1/n$-by-$1/n$ squares corresponding to the edges $u_i\vw_i$,
we have
\[
 \E\left( \int_{U'\times \VW'} \ka' \right) = \frac{1}{4}\int_{U\times \VW}\ka'
  - \frac{1}{4}\int_S \ka'.
\]
Combining the last three displayed equations using the triangle
inequality, and noting that $\mu(S)=r/n^2\le 1/n$, it follows that
\begin{eqnarray*}
 \left| \E\left( \int_{U'\times \VW'} \ka' \right)
     -\frac{1}{n^2}\E\bb{A_n(U',\VW')}\right|
 &\ge& \frac{1}{4n^2}\sum_i A_{u_i\vw_i}-\frac{1}{4}\int_S \ka' - ac/100 \\
 &\ge& \frac{(an)(cn)}{4n^2} - ac/8 - ac/100 > ac/16,
\end{eqnarray*}
using \eqref{smsm}.
On the other hand, from \eqref{cnc},
\[
 \left |\int_{U'\times \VW'} \ka' - \frac{A_n(U',\VW')}{n^2} \right| \le ac/25
\]
always holds, which implies a corresponding upper bound on the difference
of the expectations. Since $ac/25<ac/16$, we obtain a contradiction, completing the
proof.
\end{proof}

\subsection{Tree integrals and the cut metric}\label{ss_ti}

In this subsection we shall show that a certain function of a kernel
whose role will become clear later is continuous (in fact Lipschitz) with
respect to the cut metric. Here there is no particular
reason to consider only the standard ground space; instead
we consider an arbitrary probability space.

Let $(\sss,\F,\mu)$ be a probability space. 
Let $\cw$ be the set of all integrable non-negative
functions $W:\sss\times\sss\to\ooo$, and
let $\cws$ be the subset of symmetric functions.
The integrability assumption is for convenience only; the results
extend to arbitrary measurable non-negative functions if one is
a little careful with infinities in the proofs. However, we shall
only need the integrable case.

For $W\in\cw$, let
\begin{equation}\label{glw}
  \glw(x)\=\int_\sss W(x,y)\dd\mu(y)
\end{equation}
and
\begin{equation}\label{glp}
  \gl'_W(y)\=\int_\sss W(x,y)\dd\mu(x)
\end{equation}
denote the marginals of $W$; we allow the
value $+\infty$, although by
our assumption that $W$ is integrable, $\glw(x)<\infty$ a.e.\
and $\gl'_W(y)<\infty$ a.e.
Note that $\glw$ and $\gl'_W$ are measurable functions
from $\sss$ to $[0,\infty]$. 

Throughout this subsection we work with \eqref{cntdef}
as the definition of the cut norm: if $W\in L^1(\sss^2)$, then
\begin{equation}\label{cutnorm}
  \cn W \=\sup_{\sn{f}\le1,\,\sn{g}\le1}
\Bigabs{\int_{\sss^2} f(x)g(y) W(x,y)\dd\mu(x)\dd\mu(y)}.
\end{equation}

It is immediate from the definition \eqref{cutnorm} that
\begin{equation}
  \label{cutnorm1}
\cutnorm{W} \le \norm{W}\liss
\end{equation}
and that,
for any bounded functions $h$
and $k$ on $\sss$,
\begin{equation}
  \label{cutnorm2}
\cutnorm{h(x)k(y)W(x,y)} \le \sn{h}\sn{k}\cutnorm{W}.
\end{equation}

Before stating the main result of this subsection, let us note that
if two kernels are close in cut norm, then their marginals are close
in $L^1$. (This is doubtless well known, but in any case very
easy to see.)

\begin{lemma}\label{L0}
If\/ $W_1,W_2\in\cw$, then 
$\norm{\gl_{W_1}-\gl_{W_2}}\lis\le\cutnorm{W_1-W_2}$.
\end{lemma}

\begin{proof}
If  $f\in L^\infty(\sss)$, then
\begin{equation*}
  \int_\sss\bigpar{\gl_{W_1}(x)-\gl_{W_2}(x)}f(x)\dd\mu(x)
=\int_{\sss^2} f(x)\bigpar{ W_1(x,y)-W_2(x,y)}\dd\mu(x)\dd\mu(y)
\end{equation*}
and the result follows from \eqref{cutnorm}, 
letting $g(y)=1$ and
taking the supremum over all $f$ with $\sn{f}\le 1$.
(Or simply taking $f(x)$ equal to the sign of $\gl_{W_1}(x)-\gl_{W_2}(x)$.)
\end{proof}

We now turn to the integrals we shall consider, one for each finite graph $F$.
Given a finite graph $F$ with vertex set \set{1,\dots,r} and $W\in \cws$, let
\begin{equation}\label{tx}
  \tx(F,W)
\=
\int_{\sss^r} \prod_{ij\in E(F)} W(x_i,x_j) 
\prod_{k=1}^{r} e^{-\glw(x_k)}
\dd \mu(x_1)\dots \dd \mu(x_r).
\end{equation}
The reason for the notation is that $\tx(F,W)$ corresponds roughly
to $1/n$ times the expected number of isolated copies of $F$ in a certain random graph
defined from $W$.

Our aim in this subsection is to prove the following result.

\begin{theorem}\label{th_FW}
Let $F$ be a tree. Then $W\mapsto \tx(F,W)$ is a bounded map on $\cws$
that is
Lipschitz continuous in the cut norm. In other words, there exists a
constant $C$ (depending on $F$ only) such that $\tx(F,W)\le C$ for all
$W\in\cws$, and $|\tx(F,W)-\tx(F,W')|\le C \cn{W-W'}$ for all
$W,W'\in\cws$.
\end{theorem}

We shall prove \refT{th_FW} via a sequence of lemmas. The first step
will
be to transform \eqref{tx} to an integral of a product over edges
only, rather than over edges and vertices. This will involve
considering asymmetric kernels, as well as different kernels
for different edges of $F$.

Given a tree $F$ with $r$ vertices
in which each edge has an arbitrary direction,
and for every edge $ij\in F$ a (not necessarily symmetric)
kernel $W\sij\in \cw$, set
\begin{equation}\label{ty}
 \ty\bigpar{F,(W\sij)_{ij\in E(F)}}
\=
\int_{\sss^r} \prod_{ij\in E(F)} W_{ij}(x_i,x_j) 
\dd \mu(x_1)\dots \dd \mu(x_r).
\end{equation}
Note that the exponential factors
$e^{-\glw(x_k)}$ present in \eqref{tx} are missing from \eqref{ty}.

We shall reintroduce the exponential factors by attaching them
to the kernels $W\sij$. Recalling the definitions
of the marginals $\la_W$ and $\la_W'$ in \eqref{glw} and \eqref{glp},
for real $a,b\ge0$ let
\begin{equation}\label{Wabdef}
  \wab(x,y)\=e^{-a\glw(x)} W(x,y) e^{-b\glw'(y)}.
\end{equation}
Finally, let $d_i$ be the (total) degree of vertex $i$ in $F$. Then,
comparing \eqref{tx} and \eqref{ty}, for every symmetric $W:\sss^2\to\ooo$ we have
\begin{equation}\label{txty}
  \tx(F,W)
=
\ty\bigpar{F,(W^{(1/d_i,1/d_j)})\sij}.
\end{equation}

To study $\tx(F,W)$, we shall first study the map $W\mapsto W^{(a,b)}$,
and then study the behaviour of $\ty$ on the restricted set of asymmetric kernels
that arise as images of this map.

\begin{lemma}\label{L1}
For every fixed $a,b\ge0$, the map $W\mapsto\wab$ is Lipschitz continuous
on $\cw$  in the cut norm; more precisely,
\begin{equation*}
\cn{W_1\ab-W_2\ab}\le 7
\cn{W_1-W_2}  
\end{equation*}
for all $W_1,W_2\in\cw$. Also, for every
$W\in\cw$, $\sup_x\gl_{\wab}(x)\le e^{-1}/a$ and $\sup_y\gl'_{\wab}(y)\le e^{-1}/b$.
\end{lemma}
Surprisingly, this turns out to be the hardest part of the proof of \refT{th_FW}. 
\begin{proof}
Let us start with the final inequalities, which are immediate consequences of the inequality
$t e^{-t}\le e\qw$. Indeed,
  \begin{equation*}
	\begin{split}
\gl_{\wab}(x)
&\=\int_\sss\wab(x,y)\dd\mu(y)	
\le \int_\sss
e^{-a\glw(x)}W(x,y)\dd\mu(y)	
\\&
=e^{-a\glw(x)}\glw(x)\le e\qw/a,	  
	\end{split}
  \end{equation*}
and similarly $\gl'_{\wab}(y)\le e^{-1}/b$.

Turning to the main assertion, let $W_1,W_2\in \cw$.
To simplify the notation  set  $\gl_j\=\gl_{W_j}$ and $\gl'_j\=\gl'_{W_j}$
for $j=1,2$.
It will turn out that we have to argue separately according to which
of $\gl_1(x)$ and $\gl_2(x)$ is larger, and similarly for $\gl'_1(y)$ and $\gl'_2(y)$.
Accordingly, define the indicator functions
\begin{align*}
  I_1(x)&\=\ett{\gl_1(x)\le\gl_2(x)}, &
  I_2(x)&\=\ett{\gl_1(x)>\gl_2(x)}, \\
  I'_1(y)&\=\ett{\gl'_1(y)\le\gl'_2(y)}, &
  I'_2(y)&\=\ett{\gl'_1(y)>\gl'_2(y)},
\end{align*}
so $I_1(x)+I_2(x)=I_1'(y)+I_2'(y)=1$.

We may write $W_1\ab-W_2\ab$, a difference of two three-term products,
as a telescopic sum of three terms in the usual way. In particular, we have
\begin{equation}\label{part1}
  \begin{split}
W_1\ab-W_2\ab
&=
\bigpar{e^{-a\gl_1(x)}-e^{-a\gl_2(x)}}e^{-b\gl'_1(y)}W_1(x,y)
\\&\qquad\qquad
+e^{-a\gl_2(x)}\bigpar{e^{-b\gl'_1(y)}-e^{-b\gl'_2(y)}}W_1(x,y)
\\&\qquad\qquad
+e^{-a\gl_2(x)}e^{-b\gl'_2(y)}\bigpar{W_1(x,y)-W_2(x,y)}.
\end{split}
\end{equation}
It will turn out that this decomposition is only useful when $\gl_1(x)\le \gl_2(x)$
and $\gl_1'(y)\le\gl_2'(y)$, so we shall multiply by the indicator
function $I_1(x)I_1'(y)$.

To bound the final term in \eqref{part1},
note that $0\le I_1(x)e^{-a\gl_2(x)}\le 1$ and
$0\le I_1'(y)e^{-a\gl_2'(y)}\le 1$, so
from \eqref{cutnorm2} we have
\begin{equation}\label{term3}
\bignorm{
I_1(x)I_1'(y)e^{-a\gl_2(x)}e^{-b\gl'_2(y)}\bigpar{W_1(x,y)-W_2(x,y)}
}\cut
\le
\cutnorm{W_1-W_2}.  
\end{equation}

For the remaining terms we estimate the $L^1$ norm, recalling
\eqref{cutnorm1}. 
Turning to the first term,
by the mean value theorem, if $\la_1(x)\le \la_2(x)$ then for
some $y\in [\la_1(x),\la_2(x)]$ we have
\[
 e^{-a\la_1(x)}-e^{-a\la_2(x)} = a|\la_1(x)-\la_2(x)|e^{-ay}
 \le a|\la_1(x)-\la_2(x)|e^{-a\la_1(x)},
\]
where $\la_1(x)\le \la_2(x)$
is used in the final inequality. It follows that
\begin{equation*}
  I_1(x)\bigabs{e^{-a\gl_1(x)}-e^{-a\gl_2(x)}}
\le a |\gl_1(x)-\gl_2(x)|e^{-a\gl_1(x)}.
\end{equation*}
Thus,
\begin{equation*}
  \begin{split}
\bignorm{ I_1(x)I_1'(y)
&\bigpar{e^{-a\gl_1(x)}-e^{-a\gl_2(x)}}e^{-b\gl'_1(y)}W_1(x,y)}\liss
\\
&\le
\bignorm{
a\lrabs{\gl_1(x)-\gl_2(x)}e^{-a\gl_1(x)}W_1(x,y)}\liss
\\&
=
\int_{\sss^2} a\lrabs{\gl_1(x)-\gl_2(x)}e^{-a\gl_1(x)}W_1(x,y)\dd\mu(y)\dd\mu(x)
\\&
=
\int_\sss a\lrabs{\gl_1(x)-\gl_2(x)}e^{-a\gl_1(x)}\gl_1(x)\dd\mu(x)
\\&
\le
e\qw\int_\sss \lrabs{\gl_1(x)-\gl_2(x)}\dd\mu(x)	
=e\qw\norm{\gl_1-\gl_2}\lis
\\&
\le e\qw\cutnorm{W_1-W_2},
  \end{split}
\end{equation*}
where we used $te^{-t}\le e^{-1}$ for the second last step and \refL{L0}
for the final step.

Similarly, for the second term in \eqref{part1} we obtain the bound
\[
\bignorm{ I_1(x)I_1'(y) e^{-a\gl_2(x)}\bigpar{e^{-b\gl'_1(y)}-e^{-b\gl'_2(y)}}W_1(x,y)}\liss
\le e\qw\cutnorm{W_1-W_2}.
\]
Putting these two bounds together with \eqref{term3}, comparing with \eqref{part1} we see
that
\begin{equation}\label{nc1}
 \Bignorm{ I_1(x)I_1'(y)\bigpar{W_1\ab(x,y)-W_2\ab(x,y)}}\cut\le (1+2e\qw)\cn{W_1-W_2}.
\end{equation}

So far we treated the case $\la_1(x)\le \la_2(x)$, $\la_1'(y)\le \la_2'(y)$.
The remaining three cases are treated similarly.

More precisely,
for $\la_1(x)\le\la_2(x)$, $\la_1'(y)>\la_2'(y)$, we use
\begin{equation*}
  \begin{split}
W_1\ab-W_2\ab
&=\bigpar{e^{-a\gl_1(x)}-e^{-a\gl_2(x)}}e^{-b\gl'_1(y)}W_1(x,y)
\\&\qquad\qquad
+e^{-a\gl_2(x)}e^{-b\gl'_1(y)}\bigpar{W_1(x,y)-W_2(x,y)}
\\&\qquad\qquad
+e^{-a\gl_2(x)}\bigpar{e^{-b\gl'_1(y)}-e^{-b\gl'_2(y)}}W_2(x,y)
\end{split}
\end{equation*}
in place of \eqref{part1} to prove the equivalent of \eqref{nc1}
with $I_1(x)I_2'(y)$ in place of $I_1(x)I_1'(y)$.

For $\la_1(x)>\la_2(x)$, $\la_1'(y)\le\la_2'(y)$ we use
\begin{equation*}
  \begin{split}
W_1\ab-W_2\ab
&=e^{-a\gl_1(x)}\bigpar{e^{-b\gl'_1(y)}-e^{-b\gl'_2(y)}}W_1(x,y)
\\&\qquad\qquad+
e^{-a\gl_1(x)}e^{-b\gl'_2(y)}\bigpar{W_1(x,y)-W_2(x,y)}
\\&\qquad\qquad+
\bigpar{e^{-a\gl_1(x)}-e^{-a\gl_2(x)}}e^{-b\gl'_2(y)}W_2(x,y)
\end{split}
\end{equation*}
to obtain a bound with $I_2(x)I_1'(y)$ as the indicator function.

Finally, for $\la_1(x)>\la_2(x)$, $\la_1'(y)>\la_2'(y)$ we use
\begin{equation*}
  \begin{split}
W_1\ab-W_2\ab&=
e^{-a\gl_1(x)}e^{-b\gl'_1(y)}\bigpar{W_1(x,y)-W_2(x,y)}
\\&\qquad\qquad+
\bigpar{e^{-a\gl_1(x)}-e^{-a\gl_2(x)}}e^{-b\gl'_1(y)}W_2(x,y)
\\&\qquad\qquad+
e^{-a\gl_2(x)}\bigpar{e^{-b\gl'_1(y)}-e^{-b\gl'_2(y)}}W_2(x,y)
\end{split}
\end{equation*}
for $I_2(x)I_2'(y)$.

The key point is that in all cases, when we
come to apply the bound obtained from the mean value theorem,
when dealing with a term $e^{-a\la_1(x)}-e^{-a\la_2(x)}$ 
we obtain a bound involving $e^{-\la_i(x)}$ for $i=1$ or $2$
depending on which of $\la_1(x)$ and $\la_2(x)$ is larger.
For the rest of the argument to work, it is important that the term
we consider contains a factor $W_i(x,y)$ rather than $W_{3-i}(x,y)$.
Similar comments apply to the  $e^{-b\la_1'(y)}-e^{-b\la_2'(y)}$ terms.
Fortunately, we can ensure that this is always the case, as shown
by the decompositions above. Informally speaking, we simply choose the right moment
to switch from $W_1$ to $W_2$.

Combining \eqref{nc1} and its equivalents, noting that
$I_1(x)I_1'(y)+I_1(x)I_2'(y)+I_2(x)I_1'(y)+I_2(x)I_2'(y)=1$, we see that
\begin{equation*}
\cn{W_1\ab-W_2\ab}\le 
(4+8e\qw)
\cn{W_1-W_2} \le 7\cn{W_1-W_2}.
\qedhere
\end{equation*}
\end{proof}

\begin{remark}
Although we do not care about the constant, let us note that the four
estimates \eqref{term3} above can be combined into a single application
of \eqref{cutnorm2},
with $h(x)=I_1(x)e^{-\la_2(x)}+I_2(x)e^{-\la_1(x)}$
and $k(y)=I_1'(y)e^{-\la_2'(y)}+I_2'(y)e^{-\la_1'(y)}$.
This gives $1+8e\qw<4$ in place of $4+8e\qw$.
\end{remark}

We next turn to the study of $\ty(F,\cdot)$ as defined by \eqref{ty},
restricting our attention to kernels with bounded marginals.
It turns out that we must first study a related function $\tz$,
which may be seen as a rooted version
of $\ty$.

Given a rooted directed graph $F$ with vertex set $\{1,2,\ldots,r\}$
and root 1, and functions $W\sij\in \cw$, let
\begin{equation*}
 \tz\bigpar{F,(W\sij)_{ij\in E(F)};x_1}
\=
\int_{\sss^{r-1}} \prod_{ij\in E(F)} W_{ij}(x_i,x_j) 
\dd \mu(x_2)\dots \dd \mu(x_r).
\end{equation*}
Note that this is a function of $x_1\in\sss$, and that 
\begin{equation}\label{erika}
   \ty\bigpar{F,(W\sij)_{ij\in E(F)}}
=
\int_\sss \tz\bigpar{F,(W\sij)_{ij\in E(F)};x}\dd\mu(x).
\end{equation}

Let $\cwa\=\set{W\in\cw:\sup_x\glw(x),\,\sup_y\glw'(y)\le \BB}$.

\begin{lemma}
  \label{L2root}
Let $F$ be a rooted directed tree and
$(W\sij)_{ij\in E(F)}$ a family with $W\sij\in\cwa$ for all $ij$.
Then for all $x\in\sss$,
\begin{equation*}
  \tz\bigpar{F,(W\sij)_{ij\in E(F)};x} \le \BB^{e(F)}.
\end{equation*}
\end{lemma}

\begin{proof}
A simple induction on the number $e(F)$ of edges of $F$. If $e(F)=0$, so
$F$ consists of just a single vertex,
then both sides are equal to $1$. For $e(F)>0$, pick a leaf $v$ of $F$
that is not the root, with neighbour $w$. 
We may assume without loss of generality that the edge $wv$ is oriented
from $w$ to $v$.
In the integrand appearing in the left hand side above, there is
only one factor that depends on $x_v$, namely $W_{wv}(x_w,x_v)$.
Integrating out over $x_v$,
this integrates to $\gl_{W_{wv}}(x_w)$.
Replacing $\gl_{W_{wv}}(x_w)$ by $\BB$, which is an upper bound by assumption,
we see that
that $\tz(F,\cdot;x) \le \BB\tz(F-v,\cdot;x)$, and the result follows by induction.
\end{proof}

Returning to the unrooted case, we are now ready
for the final step in the proof of \refT{th_FW}.

\begin{lemma}
  \label{L2}
Let $F$ be a directed tree, and $\BB<\infty$ a constant.
For all families $(W\sij)_{ij\in E(F)}$ and 
$(W'\sij)_{ij\in E(F)}$ with $W\sij,W'\sij\in\cwa$, we have
\begin{equation}\label{l2a}
  \ty\bigpar{F,(W\sij)_{ij\in E(F)}} \le \BB^{e(F)}
\end{equation}
and
\begin{equation}\label{l2b}
\bigabs{ \ty\bigpar{F,(W\sij)_{ij\in E(F)}} 
- \ty\bigpar{F,(W'\sij)_{ij\in E(F)}} }
\le \BB^{e(F)-1}\sum_{ij\in E(F)} \cn{W\sij-W'\sij}.
\end{equation}
\end{lemma}

\begin{proof}
The bound \eqref{l2a} is immediate from \eqref{erika} and \refL{L2root} by
choosing an arbitrary root.  

For the Lipschitz estimate \eqref{l2b}, it suffices to treat the case
where the families $W\sij$ and $W'\sij$ differ only on a single
edge $ij$, say $ij=12$. In
this case, let $F_1$ and $F_2$ be the two components of
$F\setminus\set{12}$, and regard these as rooted trees with roots 1
and 2, respectively. Then, simplifying the notation,
\begin{equation*}
   \ty\bigpar{F,(W\sij)_{ij}}
=\int_{\sss^2} \tz(F_1;x_1)\tz(F_2;x_2) W_{12}(x_1,x_2)
\dd\mu(x_1)\dd\mu(x_2)
\end{equation*}
and similarly for $(W'\sij)$. Thus, by \eqref{cutnorm},
	\begin{multline*}
\bigabs{ \ty\bigpar{F,(W\sij)_{ij}}
-
 \ty\bigpar{F,(W'\sij)_{ij}}}
\\=
\lrabs{
\int_{\sss^2} \tz(F_1;x_1)\tz(F_2;x_2) 
\bigpar{W_{12}(x_1,x_2)-W'_{12}(x_1,x_2)}
\dd\mu(x_1)\dd\mu(x_2)}
\\
\le \sn{\tz(F_1)}\sn{\tz(F_2)}\cutnorm{W_{12}-W'_{12}}.
\end{multline*}
The result follows by \refL{L2root}.
\end{proof}

Putting the pieces together, \refT{th_FW} follows.
\begin{proof}[Proof of \refT{th_FW}]
In the light of \eqref{txty}, this is immediate from Lemmas~\ref{L1} and~\ref{L2}.
\end{proof}

\subsection{Small components}\label{ss_s}

Let $N_k(G)$ denote the number of vertices of a graph $G$ in components
of order $k$, and let $\rho_k(\ka)$ denote the probability that $\bpk$
consists of exactly $k$ particles in total. Our next aim is to prove
the following lemma. Recall that $A_n$ is always assumed to be $n$-by-$n$.

\begin{lemma}\label{l_Nk}
Let $(A_n)$ be a sequence of non-negative symmetric matrices converging
in $\dcut$ to a kernel $\ka$, and let $k\ge 1$ be fixed.
Then $\E N_k(G(A_n))/n \to \rho_k(\ka)$.
\end{lemma}

As usual in sparse random graphs, the dominant contribution
will be from tree components.
We start with a simple lemma showing that
cyclic components can be neglected.

Let us call a sequence $(A_n)$ of non-negative symmetric matrices 
(in which $A_n$ is $n$-by-$n$ as usual) {\em well behaved}
if all the diagonal entries are zero, and
$\max A_n=o(n)$, where $\max A_n$ is the largest entry in $A_n$.
One useful property of such sequences is that for them, the models $G(A_n)$
and $\Gp(A_n)$ are essentially equivalent, as shown by the following simple lemma.
\begin{lemma}\label{GGp}
Let $\ka$ be a kernel and let $(A_n)$ be a sequence of well-behaved matrices with $\dcut(A_n,\ka)\to 0$.
Let $A_n'$ be the matrix with entries defined by \eqref{cve}.
Then $\dcut(A_n',\ka)\to 0$.
\end{lemma}
\begin{proof}
For $n$ large enough that $\max a_{ij}\le n/2$, say, from \eqref{cve} we have
$|a_{ij}-a_{ij}'| =O(a_{ij}^2/n)$, with the implicit constant $C$ absolute.
It follows that
\[
 \sum_{ij} |a_{ij}-a_{ij}'| \le C \sum_{ij} a_{ij}^2/n \le C\max\{a_{ij}/n\} \sum_{ij} a_{ij}
 =o(1)\sum_{ij} a_{ij},
\]
using the well-behavedness assumption. Since $\dcut(A_n,\ka)\to 0$, we have
$\sum a_{ij}\sim n^2\int\ka=O(n^2)$.
Hence
\[
 \dcut(\ka_{A_n},\ka_{A_n'}) \le \on{\ka_{A_n}-\ka_{A_n'}} = n^{-2}\sum_{ij}|a_{ij}-a_{ij}'|
 = o(1),
\]
and the result follows.
\end{proof}
The point of Lemma~\ref{GGp} is that if we can prove that $\Gp(A_n)$ has a certain property
whenever $\dcut(A_n,\ka)\to 0$, then the same result for $G(A_n)$ follows: we simply
express $G(A_n)$ as $\Gp(A_n')$ as in \eqref{convert},
and apply our result for $\Gp(\cdot)$ to the sequence
$(A_n')$.

Our next lemma shows that the graphs we consider have few vertices in small components
containing cycles.
Let $\nkt(G)$ denote the
number of vertices of a graph $G$ in tree components of order $k$,
and $\nkc(G)$ the number in cyclic components 
of order $k$, so $N_k(G)=\nkt(G)+\nkc(G)$.

\begin{lemma}\label{l_Nkp}
Let $(A_n)$ be a sequence of well-behaved matrices and $k\ge 2$
an integer. Then $\E \nkc(G_n) =o(n)$, where $G_n=\Gpm(A_n)$.
\end{lemma}
Note that in this lemma there is no convergence assumption. Note also that \refL{l_Nkp}
immediately implies a corresponding result for $\Gp(A_n)$, which is simply
the simple graph underlying $\Gpm(A_n)$, and so satisfies
$\nkc(\Gp(A_n))\le \nkc(\Gpm(A_n))$.
It also implies a corresponding result for $G(A_n)$; this may be
deduced from the result for $\Gp(A_n)$ by expressing $G(A_n)$ as $\Gp(A_n')$ as above.
\begin{proof}
We shall consider an evolving version $G_n(t)$ of $G_n$.
To define this, for each possible edge $ij$, construct a Poisson process on $[0,1]$
with intensity $a_{ij}/n$; the points of these processes will
be the birth times of the $ij$ edges. Let $G_n(t)$ be the graph
formed by all edges born by time $t$, noting that
the number of $ij$ edges in $G_n(1)$ is Poisson with mean
$a_{ij}/n$. Taking the processes independent,
$G_n(1)$ thus has the distribution of $G_n=\Gpm(A_n)$.

Let $M_{\le k}(G)$ denote
the number of cyclic components of a (multi-)graph $G$ of order at most $k$; thus
$\nkc(G)\le k M_{\le k}(G)$.

Let $f(t)$ denote the expectation of $M_{\le k}(G_n(t))$; then $f(0)=0$
and $f(1)=\E M_{\le k}(G_n)$, so $\E \nkc(G_n)\le  kf(1)$,
and it suffices to show that the derivative of $f$ is bounded above by $o(n)$.
Condition on $G_n(t)$, and consider the edges born in a short
time interval $[t,t+\dt]$. Taking $\dt$ small enough, the probability
that there is more than one such edge in any interval $[t,t+\dt]$
is negligible.
The only way we can have $M_{\le k}(G+e)\ge M_{\le k}(G)$ is
if $e$ joins two vertices $i$, $j$ in some component of $G$
of order at most $k$. There are at most $kn$ such pairs
of vertices. Since the $a_{ij}$ are
uniformly bounded by $o(n)$,
the probability $a_{ij}\dt/n$ of adding $e=ij$
is $o(\dt)$, and the probability of adding some such edge is $o(kn\dt)=o(n\dt)$.
Adding such an edge increases $M_{\le k}$ by at most $1$,
so the expected increase in time $\dt$ is at most $o(n\dt)$
as required.
\end{proof}

We are now ready to prove \refL{l_Nk}.

\begin{proof}[Proof of \refL{l_Nk}.]
We claim that it suffices to prove the lemma under the assumption
that $(A_n)$ is well behaved, i.e., $\max A_n=o(n)$, and
the diagonal entries are 0.

To see this, note that by Lemma~\ref{l_on} there is some $\delta=\delta(n)\to 0$ such that
at most $\delta n$ entries of $A_n$ exceed $\delta n$, and the sum
of these entries is at most $\delta n^2$.
Define $A_n'=(a_{ij}')$ by setting $a_{ij}'=0$ if $a_{ij}>\delta n$ or if
$i=j$, and setting $a_{ij}'=a_{ij}$ otherwise.
Then
\[
 \dcut(A_n,A_n') \le \frac{1}{n^2}\sum |a_{ij}-a_{ij}'|
 = \frac{1}{n^2} \sum_{a_{ij}>\delta n} a_{ij} +\frac{1}{n^2}\sum_{i:a_{ii}\le \delta n} a_{ii}
 \le \delta+\delta =o(1).
\]
Hence $\dcut(A_n',\ka) \to 0$, so the sequence $A_n'$ and kernel
$\ka$ satisfy the assumptions of the lemma, and $(A_n')$ is well behaved.
In establishing our claim we may thus
assume that 
\begin{equation}\label{n'}
 \E N_k(G(A_n'))/n\to \rho_k(\ka).
\end{equation}

But then the same result for $G(A_n)$ follows almost immediately. 
Indeed, we may assume that $G(A_n')\subset G(A_n)$, and we have
\[
 \E\bigpar{E(G(A_n))\setminus E(G(A_n'))} = \E\bigpar{e(G(A_n))-e(G(A_n'))}
 \le \frac{1}{n}\sum |a_{ij}-a_{ij}'| = o(n).
\]
Since adding an edge to a graph $G$ changes $N_k(G)$ by at most $2k$, it follows
that
\[
 \E | N_k(G(A_n))-N_k(G(A_n')) | =o(n),
\]
which with \eqref{n'} proves the same statement for $A_n$, establishing the claim.

From now on we suppose as we may that $(A_n)$ is well behaved. 
In the light of \refL{GGp} we may work with $\Gp(A_n)$ instead
of $G(A_n)$. In fact, we shall work with $G_n=\Gpm(A_n)$, which has
exactly the same component structure as $\Gp(A_n)$.

Given a loopless
multi-graph $F$ on $[k]$ and a sequence 
$\bv=(v_1,\ldots,v_k)$ with $1\le v_i\le n$
for each $i$,
set
\begin{equation}\label{p}
 p_{\bv}(F)=p_{\bv}(F,A_n) = 
  \prod_{ij\in E(F)} \frac{a_{v_iv_j}}{n} \prod_{uw: \{u,w\}\cap \{v_i\}\ne \emptyset} e^{-a_{uw}/n},
\end{equation}
where the second product is over all edges $uw$ of the complete
graph on $[n]$ meeting $\{v_1,\ldots,v_k\}$.

Let us call a sequence $\bv=(v_1,\ldots,v_k)$ {\em good} if the $v_i$ 
are distinct, and {\em bad} otherwise.
If $F$ is a simple graph and $\bv$ is good, then $p_{\bv}(F)$ is
the probability
that the vertices $v_1,\ldots,v_k$ of
$G_n=\Gpm(A_n)$
form a component isomorphic to $F$, with the $i$th vertex of $F$ mapped to $v_i$.
Hence, writing $n_F(G_n)$ for the number of components
of $G_n$ isomorphic to $F$, for simple $F$ we have
\[
 \E n_F(G_n) = \frac{1}{\aut(F)}\sum_{\bv\mathrm{\ good}} p_{\bv}(F).
\]
Our aim is to relate this sum with $F$ a tree to $\tx(F,\ka_{A_n})$, and hence to $\tx(F,\ka)$.

Let $\gl_\ka(x)$ denote the marginal of $\ka$, defined by \eqref{glw}.
For $1\le i\le n$, set
\[
 \gl_n(i) = \frac{1}{n} \sum_j a_{ij},
\]
so $\gl_n$ is essentially the marginal of $\ka_{A_n}$. (More precisely,
$\gl_n(i)$ gives the value of the marginal of $\ka_{A_n}$ at any point
of the interval of length $1/n$ corresponding to vertex $i\in [n]$.)

Given a multi-graph $F$ and a (not necessarily good) sequence $\bv$, let
\begin{equation}\label{p0}
 p^0_{\bv}(F)=p^0_{\bv}(F,A_n) = 
  \prod_{ij\in E(F)} \frac{a_{v_iv_j}}{n} \prod_{i=1}^k e^{-\la_n(v_i)}.
\end{equation}
Expanding each term $\la_n(v_i)$ and then
comparing \eqref{p} and \eqref{p0}, we see that if $\bv$ is good
then the only difference
is that certain factors $\exp(-a_{uw}/n)$ appear twice in \eqref{p0}
and only once in \eqref{p}, namely such factors with $u,w\in \{v_1,\ldots,v_k\}$.
Since there are $\binom{k}{2}=O(1)$ such factors and each is (by our
well-behavedness assumption) $1+o(1)$, we have
\begin{equation}\label{pp0}
 p_{\bv}(F) \sim p^0_{\bv}(F)
\end{equation}
uniformly in good sequences $\bv$. Hence, for simple $F$,
\begin{equation}\label{Fcount}
 \E n_F(G_n) \sim \frac{1}{\aut(F)}\sum_{\bv\mathrm{\ good}} p^0_{\bv}(F).
\end{equation}
Specializing now to the case of a tree $T$ on $[k]$,
recalling \eqref{tx} we have
\[
 \tx(T,\ka_{A_n}) = n^{-k} \sum_{\bv} \prod_{ij\in E(T)} a_{v_iv_j} \prod_{i=1}^k e^{-\la_n(v_i)},
\]
so
\[
 \sum_{\bv} p^0_{\bv}(T) = n \tx(T,\ka_{A_n}).
\]

Our next aim is to show that
\begin{equation}\label{badsum}
 \sum_{\bv\mathrm{\ bad}} p^0_{\bv}(T) =o(n).
\end{equation}
Once we have done so, it follows from the formulae above that
\begin{equation}\label{nkt}
 \E n_T(G_n) = o(n) + (1+o(1))n \frac{\tx(T,\ka_{A_n})}{\aut(T)}.
\end{equation}

In any sequence $\bv$ contributing to \eqref{badsum}, at least one pair $v_i$, $v_j$
coincides. Since $a_{ii}=0$ for every $i$, we may assume that if $ij\in E(T)$,
then $v_i\ne v_j$. Let us fix a {\em pattern} of coincidences, i.e.,
decide for which pairs $\{i,j\}$ we have $v_i=v_j$.
The contribution to \eqref{badsum} from a given pattern may be bounded by
\begin{equation}\label{bbound}
 X(F) = \sum_{\bw\mathrm{\ good}} p^0_{\bw}(F),
\end{equation}
where $F$ is the multi-graph formed from $T$ by identifying the appropriate vertices,
and $w_1,\ldots,w_s$ runs over the distinct vertices among $v_1,\ldots,v_r$.
Indeed, the only difference is that in the contribution to \eqref{badsum}
we have factors $e^{-d_i\la_n(w_i)}$ rather than $e^{-\la_n(w_i)}$ in \eqref{bbound},
where $d_i\ge 1$ is the number of the $v_j$ that are mapped to $w_i$.

Note that $F$ is connected.
If $F$ is simple, then using \eqref{Fcount} again we have
\[
 X(F)\sim \aut(F) \E n_F(G_n) =O(n),
\]
since $n_F(G_n)\le n$. Moreover, if $F$ is simple and not a tree,
then by \refL{l_Nkp} we have $X(F)=o(n)$.

If $F$ is not simple, let $F'$ be the underlying simple graph.
Then the terms of the sums defining $F'$ and $F$ are in one-to-one correspondence,
and each term for $F'$ is the term for $F$ multiplied by
$e(F)-e(F')\ge 1$ factors of the form $a_{ij}/n$. Each such factor is $o(1)$,
so we have $X(F)=o(X(F'))$. We have just seen that $X(F')=O(n)$ for any
connected simple $F'$, so if $F$ is not simple we have $X(F)=o(n)$.

Recall that we could write the sum in \eqref{badsum}
as a sum of over $O(1)$ patterns of terms
each bounded by $X(F)$ for some graph $F$ arising from identifying some sets
of non-adjacent vertices
of $T$. Any such graph contains either a cycle or one or more multiple edges, so $X(F)=o(n)$
in all cases, establishing \eqref{badsum}. As noted above, \eqref{nkt} follows.

Recall that $\dcut(\ka_{A_n},\ka)\to 0$. By \refT{th_FW} we thus have
$\tx(T,\ka_{A_n})\to \tx(T,\ka)<\infty$, so 
\begin{equation}\label{enp}
 \E n_T(G_n) =n \tx(T,\ka)/\aut(T) +o(n).
\end{equation}

Let $\bpk\isom T$ denote the event that the branching process $\bpk$ when viewed
as a tree is isomorphic to $T$ (which implies that it has total size $k$).
We claim that
\begin{equation}\label{xk=t}
 \Pr(\bpk\isom T) = \frac{k}{\aut(T)}\tx(T,\ka).  
\end{equation}
In fact, the version of \eqref{xk=t} for a \emph{rooted} tree $T$,
which is the same except that the factor $k$ is omitted,
is easily proved using induction on $k$
(see~\cite{QRperc}),
and then \eqref{xk=t} follows
easily by summing over the different rootings of $T$. 

Hence, summing over all isomorphism types of trees on $k$ vertices,
\[
 \rho_k(\ka) = k \sum_T \frac{\tx(T,\ka)}{\aut(T)},
\]
and from \eqref{enp},
\[
 \E \nkt(G_n) =  \E\left(k\sum_T n_T(G_n)\right)
 = kn\sum_T \frac{\tx(T,\ka)}{\aut(T)} +o(n)= \rho_k(\ka)n +o(n).
\]
Since $\E \nkc(G_n)=o(n)$ by \refL{l_Nkp}, it follows
that $\E N_k(G_n)=\rho_k(\ka)n+o(n)$ as required, where $G_n=\Gpm(A_n)$.
Since $\Gpm(A_n)$ and $\Gp(A_n)$ have the same components, the corresponding
statement for $\Gp(A_n)$ follows immediately, so 
we have proved a version of Lemma~\ref{l_Nk} for the model $\Gp(\cdot)$.
As noted earlier, by Lemma~\ref{GGp}, Lemma~\ref{l_Nk} follows.
\end{proof}

\begin{lemma}\label{l_Nk2}
Let $(A_n)$ be a sequence of non-negative symmetric matrices converging
in $\dcut$ to a kernel $\ka$, and let $k\ge 1$ be fixed.
Then $N_k(G(A_n))/n \pto \rho_k(\ka)$.
\end{lemma}
\begin{proof}
As in~\cite{kernels} or~\cite{QRperc}, this extension of \refL{l_Nk}
requires almost no extra work:
simply repeat the proof of Lemma~\ref{l_Nk} but considering pairs of components
of order $k$ to show that with $N=N_k(G(A_n))$ we have $\E N^2/n^2\to \rho_k(\ka)^2$.
Since $\E N/n\to \rho_k(\ka)$ by \refL{l_Nk}, it follows that $\Var(N/n)=o(1)$,
so $N/n$ is concentrated about its mean.
\end{proof}

As in~\cite{kernels} or~\cite{QRperc} we have the following corollary,
where $N_{\ge \omega}=\sum_{k\ge \omega}N_k$.
\begin{corollary}\label{c_big}
Let $(A_n)$ be a sequence of symmetric $n$-by-$n$ matrices converging
in $\dcut$ to a kernel $\ka$.
Then whenever $\omega=\omega(n)$ tends to $\infty$ sufficiently slowly
we have
$N_{\ge \omega}(G(A_n))/n\pto \rho(\ka)$.
\end{corollary}

When we have completed the proof of \refT{th1}, it will follow
(arguing as in the proof of \refT{th2} in the reducible case)
that \refC{c_big} in fact holds for every $\go(n)\to\infty$ with
$\go(n)=o(n)$.

\subsection{Connecting the large components}\label{ss_conn}

To complete the proof of \refT{th1}
we shall use a modified form of the Erd\H os--R\'enyi `sprinkling'
argument to show that almost all vertices in `large' components are in fact
in a single component. We need a strengthened form of a lemma implicit in
Bollob\'as, Borgs, Chayes and Riordan~\cite{QRperc}. 
Before stating this, let us recall another lemma from~\cite{QRperc} (again
modified, but this time in a trivial way).
By an {\em $(a,b)$-cut} in a kernel $\ka$ we mean a partition $(A, A^\cc)$
of $[0,1]$ with $a\le \mu(A)\le 1-a$ such that $\int_{A\times A^\cc}
\ka \le b$.
\begin{lemma}\label{l_ir}
Let $\ka$ be an irreducible kernel, and let $0<a<\frac 12$ be given.
There is some $b=b(\ka,a)>0$ such that $\ka$ has no $(a,b)$-cut.
\end{lemma}
\begin{proof}
The same statement is proved in~\cite[Lemma 7]{QRperc},
but for {\em graphons}, i.e., bounded
kernels; all kernels considered in~\cite{QRperc} were bounded.
Although as it happens
we shall only use the bounded case, we may as well note
that the restriction is entirely irrelevant.
Indeed, irreducibility of
a kernel $\ka$ depends only on whether certain integrals are 0, and hence
only on the set where $\ka>0$. So if $\ka$ is irreducible,
so is the pointwise minimum $\ka'$ of $\ka$ and $1$.
If $\ka$ has an $(a,b)$-cut, then so does $\ka'$, so the result
follows from the bounded case.
\end{proof}

Here then is the key lemma that we shall need. 

\begin{lemma}\label{l_join}
Let $\ka$ be an irreducible kernel and $\delta>0$ a constant.
There are positive constants $\alpha=\alpha(\ka,\delta)$ and $c=c(\ka,\delta)$
such that
for every sequence $(A_n)$ of non-negative symmetric matrices with
$\dcut(A_n,\ka)\to 0$, for all large enough $n$ we have
\[
 \Pr(X\sim_{\alpha n} Y)\ge 1-\exp(-cn)
\]
for all disjoint $X$, $Y\subset [n]$ with $|X|$, $|Y|\ge \delta n$,
where $X\sim_k Y$ denotes the event that the graph $G(A_n)$ contains at
least $k$ vertex disjoint paths starting
in $X$ and ending in $Y$.
\end{lemma}
A version of this lemma, but with the additional condition that the kernel $\ka$ and entries
of the matrices $A_n$ are uniformly bounded, is implicit in~\cite{QRperc} (see~\cite[Lemma 4.2]{clustering}).
Although the basic strategy of the proof of Lemma~\ref{l_join} is the same as that in~\cite{QRperc},
dealing with unbounded kernels requires considerable care, so we shall write
out the proof in full.
\begin{proof} 
We write $(a_{ij})$ for the entries of $A_n$, suppressing the dependence on $n$.
As before, by \refL{l_on} we may assume that $\max a\sij=o(n)$, and in particular that
$
 a\sij\le n/100,
$
say. We may also assume that $\delta<1/10$, say. 

Throughout this proof we view $A_n$ as a (dense) weighted graph.
In particular, given sets $V$ and $W$ of vertices of $A_n$, i.e., subsets of $[n]$,
we write
\[
 e(V,W) = \sum_{v\in V}\sum_{w\in W} a_{vw}
\]
for the total edge weight from $V$ to $W$. Similarly, for $v\in [n]$ and $W\subset [n]$,
\[
 e(v,W) = e(\{v\},W) = \sum_{w\in W} a_{vw}.
\]

Let $\kam=\ka\wedge 1$ be the pointwise minimum of $\ka$ and $1$.
Since $\dcut(A_n,\ka)\to 0$, there are rearrangements $\ka_n$ of $\ka$ such that
\begin{equation}\label{c}
 \cn{\ka_{A_n}-\ka_n}\to 0.
\end{equation}
Let $\kam_n=\ka_n\wedge 1$, noting that $\kam_n$ is a rearrangement of $\kam$.

Identifying a subset of $[n]$ with the union of the corresponding
intervals of length $1/n$ in $[0,1]$, for subsets $V$ and $W$ of $[n]$ we set
\[
 e_0(V,W) = n^2 \int_{V\times W} \ka_n(x,y) \dd x\dd y
\]
and
\[
 \emi_0(V,W) = n^2 \int_{V\times W} \kam_n(x,y) \dd x\dd y.
\]
From \eqref{c} there is some $\eta(n)\to 0$ such that
\[
 \bigabs{e(V,W) - e_0(V,W)} 
= n^2\left|\int_{V\times W}(\ka_{A_n}-\ka_n)\right| \le n^2\eta(n)
\]
for all $V$ and $W$.
Since $\ka\ge\kam$, so $e_0(V,W)\ge \emi_0(V,W)$, it follows that
\begin{equation}\label{em}
 e(V,W) \ge \emi_0(V,W) - n^2\eta(n).
\end{equation}

By Lemma~\ref{l_ir} there is some $b>0$ such that $\kam$ has no $(\delta/2,b)$-cut.
We may and shall assume that $b<1/10$, say.
Since each $\kam_n$ is a rearrangement of $\kam$, no $\kam_n$ has a $(\delta/2,b)$-cut.

Fix disjoint sets $X$ and $Y$ of vertices, each of size at least $\delta n$.
Arguing as in~\cite{QRperc}, we shall
inductively define an increasing sequence $S_0$, $S_1$, $\ldots$,
$S_\ell$ of sets of vertices in a way that depends on $A_n$, $X$ and $Y$,
but not on the random graph $G(A_n)$. There will be some additional
complications due to unbounded matrix entries; it turns out we can sidestep
these with appropriate use of the inequality \eqref{em}.

We start with $S_0=X$, noting that
$|S_0|\ge \delta n$. We shall stop the sequence when $|S_t|$ first
exceeds $(1-\delta/2)n$. Thus, in defining $S_{t+1}$ from $S_t$, we
may assume that $\delta n\le |S_t|\le (1-\delta/2) n$. Since $\kam_n$ has
no $(\delta/2,b)$-cut, we have
\[
 \sum_{v\notin S_t} \emi_0(v,S_t) = \emi_0(S_t^\cc,S_t) =n^2 \int_{S_t^\cc\times S_t} \kam_n \ge bn^2.
\]
Let
\[
 T_{t+1}= \{v\notin S_t: \emi_0(v,S_t) \ge bn/2\}.
\]
Since $\kam_n\le 1$ holds pointwise, $\emi_0(v,S_t)\le |S_t|\le n$ for any $v$.
Thus
\[
 bn^2 \le \emi_0(S_t^\cc,S_t)
 \le \frac{bn}{2} \bigabs{[n]\setminus(S_t \cup T_{t+1})}
 + n |T_{t+1}|
 \le \frac{bn^2}{2} + n |T_{t+1}|.
\]
Hence $|T_{t+1}|\ge \frac{bn}{2}$.
Set $S_{t+1}=S_t\cup T_{t+1}$,
and continue the construction until we reach an $S_\ell$
with $|S_\ell|\ge(1-\delta/2) n$. Note that $\ell\le
2/b = O(1)$.

We shall now turn to the random graph $G(A_n)$, uncovering the edges between $T_t$ and
$S_{t-1}$, working backwards from $T_\ell$. It will be convenient to
set $T_0=S_0$, so $S_t=\bigcup_{j=0}^t T_j$. Since $|S_\ell|\ge
(1-\delta/2)n$, while $|Y|\ge \delta n$, the set $S_\ell$ contains
at least $\delta n/2$ vertices from $Y$. Since $S_0=T_0=X$ is
disjoint from $Y$, it follows that there is some $t_0$, $1\le
t_0\le \ell$, for which $T_{t_0}$ contains a subset $Y_0$ of $Y$ with
\[
 |Y_0| \ge \delta n/(2\ell).
\]

Next, we aim to construct a set $X_0\subset S_{t_0-1}$ with
$|X_0|\geq b|Y_0|/10$ such that every $x\in X_0$ is joined to
some $y\in Y_0$ by an edge of $G(A_n)$. In fact, we shall
look for a partial matching from $Y_0$ to $S_{t_0-1}$ of size exactly
\[
 N=b|Y_0|/10;
\]
we ignore the irrelevant rounding to integers.
Let us list the vertices of $Y_0$ as $\{y_1,\ldots,y_s\}$.
We shall test each $y_i$ in turn to see whether it has a neighbour
in $S_{t_0-1}$; the complication is that we must avoid vertices
of $S_{t_0-1}$ that are neighbours of earlier $y_j$. We shall also
stop looking for new neighbours if we already have a large enough matching.

Formally,
we inductively define subsets $Z_0,Z_1,\ldots,Z_s$ of $S_{t_0-1}$,
starting with $Z_0=\emptyset$. For $1\le i\le s$,
if $|Z_{i-1}|=N$ then we set $Z_i=Z_{i-1}$.
If $|Z_{i-1}|<N$ and $y_i$ has a neighbour
$z\in S_{t_0-1}\setminus Z_{i-1}$, we set $Z_i=Z_{i-1}\cup\{z\}$
for any such neighbour $z$. If no such neighbour exists, we set $Z_i=Z_{i-1}$.
Note that $Z_0\subset Z_1\subset \cdots\subset Z_s$ is a random sequence of sets,
and $|Z_s|\le N$.

We claim that the following statement holds deterministically:
if $n$ is large enough, then
there are at least $s/2$ values of $i$ for which
\begin{equation}\label{enough}
 e(y_i, S_{t_0-1}\setminus Z_{i-1} ) \ge bn/4.
\end{equation}

Suppose that this claim does not hold, and let $Y'\subset Y_0$
be a set of at least $s/2$ vertices $y_i$ 
for which $e(y_i,S_{t_0-1}\setminus Z_{i-1} ) <bn/4$.
Since $Z_{i-1}\subset Z_s$,
for all $y\in Y'$ we have $e(y,S_{t_0-1}\setminus Z_s) <bn/4$.
Summing over $y$, we have
\[
 e(Y',S_{t_0-1}\setminus Z_s) < bn|Y'|/4.
\]
From \eqref{em} it follows that
\[
 \emi_0(Y',S_{t_0-1}\setminus Z_s) < bn|Y'|/4 + n^2\eta(n).
\]
On the other hand, since $Y'\subset T_{t_0}$, we have
\[
 \emi_0(Y',S_{t_0-1}) \ge bn|Y'|/2.
\]
Consequently,
\[
 \emi_0(Y',Z_s) = \emi_0(Y',S_{t_0-1}) - \emi_0(Y',S_{t_0-1}\setminus Z_s)
 > bn|Y'|/4- n^2\eta(n).
\]

Since $|Y'|\ge |Y|/2=\Theta(n)$, we see that if $n$ is large enough,
then $\emi_0(Y',Z_s)\ge bn|Y'|/5$.
But $\kam$ is bounded by $1$, so
\[
 \emi_0(Y',Z_s)\le |Y'||Z_s| \le |Y'|N = |Y'|(b|Y_0|/10) < bn|Y'|/5.
\]
This contradiction establishes the claim.

Suppose that for some $i$ we have $e(y_i,S_{t_0-1}\setminus Z_{i-1})\ge bn/4$.
Then the expected number of edges of $G(A_n)$ from $y$ to $S_{t_0-1}\setminus Z_{i-1}$ is
at least $b/4$, so the probability that there is at least one such edge is at least $b/5$.

From the claim above,
and independence of edges from different vertices $y$,
it follows that unless we reach $|Z_i|=N$ at some stage, 
the number of edges in the matching we find stochastically dominates
a Binomial distribution $D$ with parameters $|Y_0|/2$ and $b/4$.
More precisely, the probability that $|Z_s|< N$
is at most the probability that $D<N$.
But $D$ has mean $|Y_0|b/8 \ge N=|Y_0|b/10$. Since $|Y_0|=\Theta(n)$,
it follows (by Chernoff's inequality)
that with probability $1-\exp(-\Theta(n))$ we have
$|Z_s|\ge N$.

In summary, with probability at least $1-\exp(-\Theta(n))$ we find
a set $X_0=Z_s$ of at least $b|Y_0|/10$ vertices of $S_{t_0-1}$
such that every $x\in X_0$ is joined to some $y=y(x)\in Y_0$ by an edge
of $G(A_n)$, with the $y(x)$ distinct.

Suppose we do find such an $X_0$.
As $|X_0|\ge b|Y_0|/10$, there is some $t_1<t_0$ such that
$Y_1=X_0\cap T_{t_1}$ contains at least $b|Y_0|/(10\ell)$
vertices. If $t_1\ge 1$ then, arguing as above, with probability
$1-\exp(-\Theta(n))$ we find a $t_2$ and a set $Y_2$ of at least
$b^2|Y_0|/(10\ell)^2$ vertices of $T_{t_2}$ joined in
$G(A_n)$ to $Y_1$, and so on. As the sequence $t_0$, $t_1,\ldots$
is decreasing, this process terminates after $s\le \ell$ steps with
$t_s=0$. Hence, with probability $1-\exp(-\Theta(n))$ we find a set
$Y_s$ of at least $(b/(10\ell))^\ell |Y_0| = \Theta(n)$
vertices of $T_0=S_0= X$ joined in $G(A_n)$ by
vertex disjoint paths to vertices in $Y$,
completing the proof of \refL{l_join}.
\end{proof}

As in~\cite{QRperc}, \refC{c_big} and \refL{l_join} easily combine to give
\refT{th1}.
\begin{proof}[Proof of \refT{th1}]
Let $G_n=G(A_n)$.
By \refC{c_big} there is some $\omega=\omega(n)$ with $\omega(n)\to\infty$
such that $N_{\ge \omega}(G_n)/n\pto \rho(\ka)$.
We may and shall assume that $\omega=o(n)$.
Since 
\[
 C_1(G_n)+C_2(G_n) 
\le \max\{2\omega, N_{\ge\omega}(G_n)+\omega\} \le \rho(\ka)n+\op(n),
\]
it suffices to prove that if $\ka$ is irreducible then
\begin{equation}\label{below}
 C_1(G_n)\ge \rho(\ka)n +\op(n).
\end{equation}
If $\rho(\ka)=0$, then this statement holds vacuously, so suppose that
$\ka$ is irreducible and
$\rho(\ka)>0$.

Fix $0<\eps<\rho(\ka)/10$.
By \cite[Theorem 6.4]{kernels} we have $\rho((1-\gamma)\ka)\upto \rho(\ka)$
as $\gamma\to 0$. Fix $0<\gamma<1$ such that $\rho((1-\gamma)\ka)>\rho(\ka)-\eps$.

Let $G_n'=G((1-\gamma)A_n)$ and $G_n''=G(\gamma A_n)$ be independent. We may
and shall assume that $G_n'\cup G_n''\subseteq G_n$.
Applying \refC{c_big} to the sequence $(1-\gamma)A_n$, which tends to $(1-\gamma)\ka$
in $\dcut$, we see that there is an $\omega=\omega(n)$ tending to infinity such that
\begin{equation}\label{Gn'}
 N_{\ge\omega}(G_n')\ge (\rho((1-\gamma)\ka)-\eps)n \ge (\rho(\ka)-2\eps) n
\end{equation}
holds whp. Let us condition on $G_n'$ assuming that \eqref{Gn'} does hold.
Let $B$ be the set of vertices of $G_n'$ in components of size at
least $\omega$ (we call these components \emph{large}),
so $|B|\ge (\rho(\ka)-2\eps)n$.

If $C_1(G_n)\le (\rho(\ka)-3\eps)n$ then there is a partition $(X,Y)$ of $B$
such that $|X|$, $|Y|\ge \eps n$, with no path in $G_n$ joining $X$ to $Y$.
Let us call such a partition {\em bad}. Since $G_n'\subset G_n$,
each of $X$ and $Y$ must be a union of large
components of $G_n'$, so there 
are at most $2^{n/\omega(n)}$ choices for $(X,Y)$.
But the probability that a given pair $(X,Y)$ is bad
is at most the probability that there is no path in $G_n''\subset G_n$ from $X$ to $Y$;
by Lemma~\ref{l_join} this probability is $\exp(-\Theta(n))$. Hence
the expected number of bad partitions is $o(1)$, and whp there is no such partition.
Thus $C_1(G_n)\ge (\rho(\ka)-3\eps)n$ whp.
Letting $\eps\to 0$, the bound \eqref{below} follows, and this is all that
is required to complete the proof of \refT{th1}.
\end{proof}

\subsection{The reducible case: proof of \refT{th2}}\label{ss_red}

In this subsection we shall justify the terminology
by showing that one can reduce the reducible case to the irreducible
case. Surprisingly, in this setting (unlike that
of~\cite{clustering}), this is not quite immediate.

The key step is a lemma allowing us to partition a sequence of
matrices converging to a reducible kernel.  By the {\em
restriction} $\ka_\sss$ of a kernel $\ka$ to a set $\sss\subset [0,1]$
we simply mean the function
obtained by restricting $\ka$ to ${\sss\times\sss}$,
which we may think of as a kernel on a
measure space that is no longer a probability space.
It will often be convenient to consider the rescaled restriction
$\ka_\sss'$: when $\sss$ is an interval (which we can always assume)
this is the kernel on $[0,1]^2$ obtained by linearly `stretching'
$\ka_\sss$ in the obvious way.

\begin{lemma}\label{split}
Let $\ka$ be a reducible kernel and $(\sss_1, \sss_2)$ a partition
of $[0,1]$ with $0<\mu(\sss_1),\mu(\sss_2)<1$ such that $\ka_{\sss_1}$
is irreducible and $\ka$ is zero a.e.\ on $\sss_1\times\sss_2$.
If $(A_n)$ is a sequence of non-negative symmetric matrices such that
$\dcut(A_n,\ka)\to 0$ then we may find for each $n$ complementary
subsets $V_{n,1}$ and $V_{n,2}$ of $[n]$ such $|V_{n,i}|\sim \mu(\sss_i) n$
and $\dcut(A_{n,i},\ka_i')\to 0$, where $\ka_i'=\ka_{\sss_i}'$ is the
rescaled restriction
of $\ka$ to $\sss_i$ and $A_{n,i}$ is the principal minor of $A_n$
obtained by selecting the rows and columns indexed by $V_{n,i}$.
Moreover, the sum of the entries of $A_n$ corresponding to $(i,j)\in V_{n,1}\times V_{n,2}$ is $o(n^2)$.
\end{lemma}

In other words, we may split the vertex set of the random graph $G(A_n)$
into $V_{n,1}$ and $V_{n,2}$ so that the corresponding random graphs
have edge probability matrices converging to the restrictions
of $\ka$ to $\sss_1$ and $\sss_2$ respectively
(after suitable rescaling).

\begin{proof}
Suppose that $\dcut(A_n,\ka)\to 0$.
Let $(\rn)$ be a sequence of measure-preserving bijections from $[0,1]$ to itself,
corresponding to rearrangements of the kernels $\ka_{A_n}$.
Let $I_{n,i}=((i-1)/n,i/n]$ denote the subinterval of $[0,1]$ corresponding
to vertex $i$, i.e., to the $i$th row/column of $A_n$.
Then, in the rearrangement, $I_{n,i}\cap \rn(\sss_j)$ is the portion
of $I_{n,i}$ that is rearranged to correspond to part of $\sss_j$.
We write
\[
 s_{n,i} = \min_{j=1,2} \mu\bigpar{I_{n,i}\cap \rn(\sss_j)}
\]
for the extent that $I_{n,i}$ is {\em split} between $\sss_1$ and $\sss_2$,
noting that $0\le s_{n,i}<\mu(I_{n,i})=1/n$.

We call the sequence $(\rn)$ {\em good} if
\begin{equation}\label{2gc}
 \cn{ \ka_{A_n}^{(\rn)} -  \ka } \to 0,
\end{equation}
and
\[ 
 s_n = \sum_{i=1}^n s_{n,i} =o(1).
\]
Such a good sequence corresponds to rearranging $A_n$ to be close
to $\ka$ in the cut norm, while mapping almost every vertex
either almost entirely into $\sss_1$ or almost entirely into $\sss_2$.
It is not too hard to check that if such a sequence exists, then the
first conclusion of the lemma follows; we omit the tedious details, noting
only that since $\ka$ is integrable, for any subsets $X_n$ of $[0,1]^2$ with
measure tending to $0$ we have $\int_{X_n}\ka\to 0$. This shows that
changing our rearrangement on a set of measure $o(1)$
will not affect cut norm convergence. To see that the final statement follows,
let $U_{n,j}$ be the subset of $[0,1]$ corresponding to $V_{n,j}$. Then
\begin{equation*}
  \begin{split}
  \int_{U_{n,1}\times U_{n,2}}\ka_{A_n}
&=  \int_{\rn\qw(U_{n,1})\times \rn\qw(U_{n,2})}\ka_{A_n}^{(\rn)} 
\\&
\le  \cn{ \ka_{A_n}^{(\rn)} -  \ka } 
+ \int_{\rn\qw(U_{n,1})\times \rn\qw(U_{n,2})}\ka 
=o(1),	
  \end{split}
\end{equation*}
since $\rn\qw(U_{n,j})$ differs from $\sss_j$ in a set of measure $o(1)$.

It remains to prove that a good sequence exists. By hypothesis, there
is a sequence $(\rn)$ such that \eqref{2gc} holds; as we shall see,
any such sequence must be good! Indeed, suppose $s_n$ does not tend to zero.
Then passing to a subsequence, we may assume that $s_n\ge \delta$ for every $n$,
for some $\delta>0$.

For every $n$ in our (sub)sequence, and each $i\in[n]$,
pick subsets $E_{i,1}, E_{i,2}$ of $I_{n,i}$ of measure
$s_{n,i}$ with $E_{i,j}\subset \rn(\sss_j)$;
this is possible by the definition of $s_{n,i}$.
Finally, for $j=1,2$, let $E_j=\bigcup_{i=1}^n E_{i,j}$, noting that $E_j$ depends
on $n$, and that $\mu(E_j)=s_n\ge \delta$.

Since $\rn^{-1}(E_2)\subset \sss_2$, 
we have $\int_{\rn^{-1}(E_2)\times \sss_1}\ka=0$.
From \eqref{2gc} and the definition of the cut norm it follows 
that $\int_{E_2\times \rn(\sss_1)} \ka_{A_n}=o(1)$.
But
\[
 \int_{E_1\times \rn(\sss_1)} \ka_{A_n}= \int_{E_2\times \rn(\sss_1)} \ka_{A_n},
\]
since $\ka_{A_n}(x,y)$ depends on $x$ only through which interval $I_{n,i}$
the point $x$ lies in, and $E_1$ and $E_2$ intersect each $I_{n,i}$
in sets of the same measure.
Hence, $\int_{E_1\times \rn(\sss_1)} \ka_{A_n}=o(1)$, and, using
\eqref{2gc} again,
$I=\int_{\rn^{-1}(E_1)\times \sss_1}\ka=o(1)$.

But $\ka_{\sss_1}$ is irreducible, so for a.e.\ $x$ in $\sss_1$
we have $f(x)=\int_{\sss_1} \ka(x,y)\dd y>0$.
It follows that there is some $\gamma>0$ such that the integral of $f$
over any subset of $\sss_1$
of measure at least $\delta$ is at least $\gamma$.
But $I$ is exactly such an integral, 
since $\rn^{-1}(E_1)\subset \sss_1$,
giving a contradiction.
This contradiction shows that $(\rn)$ is indeed good, completing the proof.
\end{proof}

Using \refL{split}, it is not hard to deduce \refT{th2} from \refT{th1}.
\begin{proof}[Proof of~\refT{th2}]
Multiplying the kernel $\ka$ by $c$, we may and shall assume that $c=1$.

Part (a) of \refT{th2} follows from
the first statement of \refT{th1};
part (c) is a restatement of the second statement of \refT{th1},
so it remains only to prove part (b).

As shown in~\cite[Lemma 5.17]{kernels}, we may decompose $\ka$ into
irreducible kernels. 
More precisely, there is a partition $(\sss_i)_{i=0}^N$
of $[0,1]$
with $0\le N\le\infty$
such that each
$\sss_i$ has positive measure,
the restriction $\ka_i$
of $\ka$ to $\sss_i\times \sss_i$ is irreducible for each $i\ge 1$,
and $\ka$ is zero a.e.\ off $\bigcup_{i=1}^N \sss_i\times\sss_i$.

By assumption, $\dcut(A_n,\ka)\to 0$. Applying~\refL{split} repeatedly,
for any finite $N'\le N$ we may
split the vertex set $[n]$ of the graph $G_n$ into $N'+1$
subsets $V_{n,i}$, $i=0,1,\ldots,N'$, such that,
for each $i\neq0$, 
$|V_{n,i}|\sim \mu(\sss_i)n$
and $\dcut(A_{n,i}',\ka_i')\to 0$,
where $A_{n,i}'$ is the submatrix of $A_n$ corresponding to $V_{n,i}$,
and $\ka_i'=\ka_{\sss_i}'$ is the rescaled restriction of $\ka$ to $\sss_i$.
Let $G_{n,i}$ be the subgraph of $G_n$ induced by $V_{n,i}$.

In what follows it is convenient to add zero rows and columns to $A_{n,i}'$ to obtain
an $n$-by-$n$ matrix $A_{n,i}$, and to consider the kernel $\ka_i$ on $[0,1]^2$
agreeing with $\ka$ on $\sss_i^2$ and equal to zero off this set.
It is easy to check that $\dcut(A_{n,i}',\ka_i')\to 0$ implies
$\dcut(A_{n,i},\ka_i)\to 0$.
Although $\ka_i$ is formally reducible, it is so only in a trivial
sense (called quasi-irreducible in \cite{kernels}),
and by rescaling suitably it is easy to check that \refT{th1} applies
to such kernels (with, as it happens, no extra factors from the rescaling),
so by \refT{th1} we have $C_1(G_{n,i})/n\pto \rho(\ka_i)$ for each $i\ge1$.

By assumption, $\norm{T_\ka}>1$. But
\begin{equation}\label{2sn}
 \norm{T_\ka}=\sup_i \norm{T_{\ka_i}},
\end{equation}
so there is some $i$ with $\norm{T_{\ka_i}}>1$. 
We choose $N'\ge i$. 
Since
$C_1(G_n)\ge C_1(G_{n,i})$, it follows that $C_1(G_n)=\Theta(n)$ whp
as claimed.
Finally, suppose that $\ka$ is bounded, by $M$, say.
Since $\norm{T_{\ka_i}}\le M\mu(\sss_i)$, 
only finitely many of the $T_{\ka_i}$ can have norm exceeding any constant,
and the supremum in \eqref{2sn} is attained, say at $i=j$.
As noted in~\cite{QRperc}, the bound
$\rho(\ka)\ge (\norm{T_{\ka}}-1)/\sup \ka$
is implicit in~\cite{kernels}. Applying this to $\ka_j$,
the final part of \refT{th2}(b) follows.
\end{proof}

Note that we cannot say what the limiting size of the giant component
is in the reducible case: we know that there are $\op(n)$ edges joining
different $G_{n,i}$, but have no further control on these edges (which
may be completely absent), so we do not know whether they 
link the largest components in the different $G_{n,i}$ or not.
Thus $C_1(G_n)/n$ may be as small as $\max_i \rho(\ka_i)+\op(1)$,
or as large as $\rho(\ka)+\op(1)=\sum_i\rho(\ka_i)+\op(1)$.

Let us close this subsection with a conjecture. By a {\em rearrangement}
$B_n$ of a matrix $A_n$ we simply mean a matrix obtained from $A_n$
by applying some permutation to the columns, and the same permutation
to the rows.

\begin{conjecture}
Let $\ka$ be a kernel, and $(A_n)$ a sequence of non-negative symmetric
matrices in which $A_n$ is $n$-by-$n$, such that $\dcut(A_n,\ka)\to 0$.
Then there exist rearrangements $B_n$ of each $A_n$ such that
$\cn{\ka_{B_n}-\ka} \to 0$.
\end{conjecture}

A proof of this conjecture would give a simpler reduction of
the irreducible case to the reducible one. We can prove versions
of this conjecture with various additional assumptions.
Suppose first that $\ka$ is of finite type.
Then the proof of~\refL{split} adapts easily to give the desired
rearrangements:
first show that in rearrangements (almost) realizing the cut distance, there
is no significant splitting of vertices between the parts of $\ka$
(unless two parts of $\ka$ are `equivalent', but then they may be united
into a single part). This leads eventually to a rearrangement mapping
almost every vertex to some subset of some part of $\ka$; since
$\ka$ is constant on its parts, the subset is irrelevant and may be
taken to be an interval, leading to the required $B_n$.

On the other hand, suppose that both $\ka$ and the entries of all
$A_n$ are uniformly bounded, without loss of generality by $1$.
Then approximating $\ka$ by some $n$-by-$n$ kernel, and
using
a result of Borgs, Chayes, Lov\'asz, S\'os and Vesztergombi~\cite{BCLSV:1}
that if two $n$-by-$n$ kernels bounded by $1$ are within distance
$\delta$ in the cut metric, then there are rearrangements
of the corresponding matrices that are within $32\delta^{1/67}$
in the cut norm, one can find $B_n$ with $\cn{B_n-\ka}\to 0$.

\subsection{Stability}\label{ss_stab} 

In this subsection we shall prove our stability result, Theorem~\ref{th_stab},
and deduce \refT{th_cb}.
As in~\cite{kernels}, we adapt an argument of Luczak and McDiarmid~\cite{LMcD}
showing that for $c>1$ constant, whp the giant component of $G(n,c/n)$ has
the property that if its vertex set is divided into two pieces
that are not too small, then there are many edges from one piece to the other.
We shall need the following deterministic lemma from \cite{LMcD}.

\begin{lemma}\label{LLMcD}
For any $\eps>0$, there exist $\eta_0=\eta_0(\eps)>0$ and $n_0$
such that the following holds. For all $n\ge n_0$, and for
all connected graphs $G$ with $n$ vertices, there are at most
$(1+\eps)^n$ bipartitions of $G$ with at most $\eta_0n$ cross
edges.\noproof
\end{lemma}

Using this and \refL{l_join}, we shall prove the following
lemma, which corresponds roughly to the edge deletion case
of \refT{th_stab}.

\begin{lemma}\label{lstab}
Let $\ka$ be an irreducible kernel and $(A_n)$
a sequence of non-negative symmetric matrices such that
$\dcut(A_n,\ka)\to 0$.
For every $\eps>0$ there is a $\delta=\delta(\ka,\eps)>0$ such that, whp,
\[
 C_1(G_n')\ge (\rho(\ka)-\eps) n
\]
for every graph $G_n'$ that may be obtained from $G(A_n)$ by deleting
at most $\delta n$ edges.
\end{lemma}
\begin{proof}
We may assume that $\rho(\ka)>0$, as otherwise there is nothing to prove.
Reducing $\eps$ if necessary, we may and shall assume that $\eps<\rho(\ka)/10$.

Let $B_\delta$ be the `bad' event that it is possible to delete at most
$\delta n$ edges from $G_n=G(A_n)$ so that in what remains no component
contains more than $(\rho(\ka)-\eps)n$ vertices; our aim
is to show that for some constant $\delta>0$ we have $\Pr(B_\delta)\to 0$.

Suppressing the dependence on $n$, given $0<\gamma<1$, let
$G_1=G((1-\gamma) A_n)$ and $G_2=G(\gamma A_n)$. As before,
taking $G_1$ and $G_2$ independent we may assume that $G_1\cup G_2\subseteq G_n=G(A_n)$.
As noted earlier,
by \cite[Theorem 6.4]{kernels} we have $\rho((1-\gamma)\ka)\upto \rho(\ka)$
as $\gamma\to 0$. Fix $0<\gamma<1$ such that $\rho((1-\gamma)\ka)>\rho(\ka)-\eps/2$.

As in \cite{LMcD}, let $U_1$ 
denote the largest component $G_1$,
chosen according to any rule if there is a tie, and consider the event
\[
 A_1 \= \{ |U_1| \ge (\rho(\ka)-\eps/2) n \}.
\]
Since $\rho((1-\gamma)\ka)>\rho(\ka)-\eps/2$, applying \refT{th1} to
$G_1$ we see that $A_1$ holds whp.

By \refL{l_join}, applied with $\gamma\ka$ in place of $\ka$, there exist constants $\alpha>0$
and $c>0$ such that, given two disjoint sets $X$, $Y$ of vertices
of $G_2$ with $|X|,|Y|\ge \eps n/2$, we have
\begin{equation}\label{ljapp}
 \Pr( X \sim_{\alpha n} Y) \ge 1 - e^{-cn}
\end{equation}
for all large enough $n$, where $X\sim_k Y$ is the event that there are
at least $k$ vertex disjoint paths from $X$ to $Y$ in $G_2$.
Let $\eta=\eta_0(c/2)$, where $\eta_0(\cdot)$ is the function
appearing in \refL{LLMcD}, and set
\[
 \delta=\min\{(\rho(\ka)-\eps/2)\eta,\alpha/2\}.
\]

Suppose that $B=B_\delta$ and $A_1$ both hold. Then there is a set
$E$ of at most $\delta n$ edges of $G_n$ such that
in $G_n'=G_n-E$ there is no component with more than
$(\rho(\ka)-\eps)n \le |U_1|-\eps n/2$ vertices. In particular,
there is a bipartition $(X,Y)$ of $U_1$ with
$|X|$, $|Y|\ge \eps n/2$ such that there
is no path in $G_n'$ from $X$ to $Y$.
But then two conditions must hold: (i) in $G_1$ there are
at most $\delta n\le \eta |U_1|$ edges from $X$ to $Y$,
and (ii) it is possible to separate $X$ from $Y$
in $G_2$ by deleting at most $\delta n< \alpha n$
edges.

Let us condition on $G_1$, assuming that $A_1$ holds.
Then by \refL{LLMcD}, if $n$ is large enough, 
there are at most $(1+c/2)^{|U_1|}\le (1+c/2)^n\le e^{cn/2}$
bipartitions $(X,Y)$ of $U_1$ with $|X|,|Y|\ge \eps n/2$
satisfying property (i).
By \eqref{ljapp}, for each of these bipartitions
the probability that it
has property (ii) is at most $e^{-cn}$.
It follows that $\Pr(B\cap A_1) \le e^{cn/2}e^{-cn}=o(1)$.
Since $A_1$ holds whp, we thus have $\Pr(B)=o(1)$, as required.
\end{proof}

To handle the deletion of vertices rather than edges we simply
show that whp all small sets of vertices meet few edges.

\begin{lemma}\label{onon}
Let $\ka$ be a kernel and $\delta>0$ a real number. Then
there is a $\gamma>0$ such that, if $(A_n)$
a sequence of non-negative symmetric matrices with
$\dcut(A_n,\ka)\to 0$, then whp every set
of at most $\gamma n$ vertices of $G(A_n)$ meets
at most $\delta n$ edges.
\end{lemma}
\begin{proof}
For $0<\gamma<1$ let $f(\ga)=\sup \int_{A\times [0,1]} \ka(x,y)\dd\mu(x)\dd\mu(y)$,
where the supremum is over all subsets $A$ of $[0,1]$ with $\mu(A) \le \gamma$.
Since $\ka$ is integrable, we have $f(\gamma)\to 0$ as $\gamma\to 0$,
and there is some $\gamma_0$ with $f(\gamma_0)<\delta/4$.
Let us fix $\gamma\le \gamma_0$ chosen small enough that $(e/\gamma)^\gamma \le e^{\delta/20}$, say.

Given a set $U$ of vertices of $G_n=G(A_n)$, let $\nu(U)$ denote the expectation
of the sum of the degrees of the vertices in $U$.
If $|U|\le\gamma n$, then from the definition of the cut metric we have
\[
 \nu(U)/n \le f(\gamma)+\dcut(A_n,\ka),
\]
so for $n$ large enough we have $\nu(U)\le \delta n/2$ for all such $U$.
The number of edges incident with $U$ has expectation at most $\nu(U)$,
and is a sum of independent indicator variables. It follows from the Chernoff
bounds that the probability that a given $U$ meets at least $\delta n$
edges is at most $e^{-\delta n/10}$, say.
Since there are at most $\binom{n}{\gamma n} \le (e/\gamma)^{\gamma n}\le e^{\delta n/20}$
choices for $U$ with $|U|=\floor{\gamma n}$, the result follows.
\end{proof}

We are now ready to prove \refT{th_stab}.
\begin{proof}[Proof of \refT{th_stab}]
Recall that $G_n'$ will be obtained from $G_n=G(A_n)$ by deleting at most $\delta n$ vertices, and then
adding and deleting at most $\delta n$ edges. Considering when
$C_1(G_n')$ is maximized or minimized, it clearly suffices to prove
that if $\delta$ is chosen small enough,
then whp $C_1(G_n')\ge (\rho(\ka)-\eps) n$ for all such $G_n'$
obtained by deletion only, and that whp $C_1(G_n')\le (\rho(\ka)+\eps) n$ for
such $G_n'$ obtained by adding edges to $G_n$.

The first statement is immediate from Lemmas~\ref{lstab} and~\ref{onon}
as in~\cite{kernels}; we omit the simple details.

The second statement follows easily Lemma~\ref{l_Nk2}; the argument
is identical to that in~\cite{kernels}. Simply choose
$k$ such that $\sum_{k'\le k}\rho_{k'}(\ka)\ge 1-\rho(\ka)-\eps/3$;
then by Lemma~\ref{l_Nk2} there are whp at least $(1-\rho(\ka)-\eps/2)n$
vertices of $G_n$ in components of size at most $k$.
Set $\delta=\eps/(4k)$, and note that adding at most $\delta n$
edges changes the number of vertices in components of size at most
$k$ by at most $2k\delta n=\eps n/2$.
\end{proof}

We now turn to the proof of Theorem~\ref{th_cb}, giving
exponential tail bounds on the size of $C_1(G_n)$.
\begin{proof}[Proof of Theorem~\ref{th_cb}]
In proving the lower bound on $C_1(G_n)$, we may assume that $\eps<\rho(\ka)$,
and in particular that $\rho(\ka)>0$.
Given a graph $G$, let $D=D(G)$ be the minimal
$d$ such that it is possible to delete $d$ vertices
from $G$ to obtain a graph $G'$ with $C_1(G')\le (\rho(\ka)-\eps)n$.
Note that if $G_1$ and $G_2$ differ only in the set of edges
incident with some vertex $v$, then $|D(G_1)-D(G_2)|\le 1$.
\refT{th_stab} implies that for some $\delta>0$ we have $\E D(G_n) \ge \delta n$ for 
all large enough $n$.
Constructing $G_n$ by making $n$ independent choices,
where the $i$th choice is the set of edges $ji$, $j<i$,
it follows from McDiarmid's inequality~\cite{McD}
that
\begin{equation}\label{eclb}
 \Pr\bb{C_1(G_n)\le (\rho(\ka)-\eps)n} = \Pr(D(G_n)=0) \le e^{-2(\delta n)^2/n} = e^{-2\delta^2n}.
\end{equation}
(Of course, one can instead use the Hoeffding--Azuma inequality, in which case
the factor two in the exponent is in the denominator.)

Turning to the upper bounds on $C_1(G_n)$ and $C_2(G_n)$,
fix $k\ge 1$ with $\rho_{\le k}(\ka)=\sum_{k'\le k}\rho_{k'}(\ka)\ge 1-\rho(\ka)-\eps/4$,
and consider $N_n=N_{\le k}(G_n)$.
We have $\E N_n/n\to \rho_{\le k}(\ka)$ by \refL{l_Nk},
so for $n$ large enough we have $\E N_n\ge (1-\rho(\ka)-\eps/3)n$.
We shall show that
\begin{equation}\label{Nc}
 \Pr\bb{ |N_n-\E N_n| \ge \eps n/2 } \le e^{-\gamma n}
\end{equation}
for some $\gamma>0$; then, for $n$ large enough,
\begin{eqnarray*}
 \Pr\bb{C_1(G_n)+C_2(G_n)\ge (\rho(\ka)+\eps)n}
 &\le& \Pr\bb{N_{>k}(G_n) +2k \ge (\rho(\ka)+\eps)n} \\
 &\le& \Pr\bb{N_n\le \E N_n-\eps n/2} \le e^{-\gamma n}.
\end{eqnarray*}
Together with \eqref{eclb} this gives the required bounds
on $C_1(G_n)$. For the bound on $C_2(G_n)$, we use
\eqref{eclb} to bound $C_1(G_n)$ from below, and replace $\eps$ by $\eps/2$.

In our proof of \eqref{Nc}
the key point is that $N_{\le k}(G)$ is edge-Lipschitz:
if $G$ and $G'$ differ in one edge,  then $|N_{\le k}(G)-N_{\le k}(G')|\le 2k$.
To prove concentration, we apply Talagrand's inequality~\cite{Talagrand}
in the form of \cite[Theorem 2.29]{JLR}. With $N=\binom{n}{2}$,
the independent variables $Z_1,\ldots,Z_N$ are the indicator
functions of the events that the individual edges are present.
Let $f(G_n)=f(Z_1,\ldots,Z_N)=n-N_n=N_{>k}(G_n)$.
Then changing one $Z_i$ changes $N_n$, and hence $f$, by at most $c_i=2k$.
Whenever $f(G_n)\ge r$, then taking (the edge set of) one spanning tree
for each component of size greater than $k$, there is a certificate
of size at most $n$ for the event that $f(G_n)\ge r$.
Hence we may take $\psi(r)=(2k)^2 n$ for all $r$, and Talagrand's inequality
gives
\[
 \Pr(|f(G_n) - m| \ge t) \le 4e^{-t^2/(16k^2n)},
\]
where $m$ is the median value of $f(G_n)$. As usual (see, e.g., \cite{JLR}),
it then follows that the mean and median are close (within $O(\sqrt n)$),
and recalling that $N_n=n-f(G_n)$, for $n$ large enough
we obtain \eqref{Nc} with $\gamma=\eps^2/(70k^2)$, say.
\end{proof}

\section{Extension to hypergraphs}\label{sec_hyp}

In this section we shall prove an extension of Theorems~\ref{th1}
and~\ref{th2} to hypergraphs. Alternatively, this may be thought of as an extension
of the random graph model with clustering introduced in~\cite{clustering}.
Most of our arguments are simple modifications of those in previous sections, so 
we shall only outline them. There are one or two places where adapting the
proof is not so easy, and there we shall give more detail.

Let $(\sss,\mu)$ be a probability space. We write $\cwr$ for the set of all
integrable non-negative functions $W:\sss^r\to [0,\infty)$,
and $\cwrs$ for the subset of such functions that are symmetric under permutations
of the coordinates. Often we shall call a function $\ka_r\in \cwrs$ an
{\em $r$-kernel}. A {\em hyperkernel} $\kf$ is simply a sequence $(\ka_r)_{r\ge 2}$,
where $\ka_r$ is an $r$-kernel. The {\em integral} $i(\kf)$ of a hyperkernel
is defined to be 
\[
 i(\kf) = \sum_{r\ge 2} r \int_{\sss^r} \ka_r,
\]
and a hyperkernel $\kf$ is {\em integrable} if $i(\kf)<\infty$.

The cut norm has a natural extension to $r$-kernels or indeed to $L^1(\sss^r)\supset \cwr$.
As before, we consider two slightly different definitions: for $W\in L^1(\sss^r)$ set
\begin{equation}\label{cnr1} 
  \cnone W \= \sup_{S_1,\ldots,S_r}
  \Bigabs{\int_{S_1\times\cdots\times S_r} W(x_1,\ldots,x_r)},
\end{equation}
where the supremum is over all $r$-tuples of measurable subsets of $\sss$.

Alternatively, we may consider
\begin{equation}\label{cnr2} 
  \cntwo W \= \sup_{\sn{f_1},\cdots,\sn{f_r}\le1}
  \Bigabs{\int_{\sss^r} f_1(x_1)\cdots f_r(x_r) W(x_1,\ldots,x_r)}.
\end{equation}
Much of the time it makes no difference which version of $\cn{\cdot}$ we consider:
as before, in the supremum in \eqref{cnr2} we may assume that each $f_i$ is a $\pm 1$ function,
and we see that
\[
 \cnone{W} \le \cntwo{W} \le 2^r \cnone{W}.
\]
While \eqref{cnr2} is the more natural definition from the point of view of functional
analysis, we shall in fact take \eqref{cnr1} as the definition for most of this section,
writing $\cn{W}$ for $\cnone{W}$ -- 
it turns out that we obtain a very slightly stronger result this way.

Given a family $\uW=(W_r)_{r\ge 2}$ with $W_r\in \cwr$, set
\[
 i(\uW) = \sum_{r\ge 2} r\int_{\sss^r} W_r,
\]
\[
 \on{\uW} = \sum_{r\ge 2} r \on{W_r},
\]
and
\begin{equation}\label{cnrs}
 \cn{\uW} = \sum_{r\ge 2} r \cn{W_r},
\end{equation}
where $\cn{\cdot}=\cnone{\cdot}$.
The reason for the factors of $r$ above will become clear shortly.

Note that while considering a single value
of $r$, it is irrelevant whether we use $\cntwo{\cdot}$ or $\cnone{\cdot}$.
However, as soon as we sum cut norms for different $r$, the potential factor of up to
$2^r$ may make a difference. All our results will apply using $\cntwo{\cdot}$
instead of $\cnone{\cdot}$, but they would then be slightly weaker, as fewer
sequences of hyperkernels converge in the resulting norm.

Note that for $W\in L^1(\sss^r)$ we trivially have
\[
 \Bigabs{\int_{\sss^r} W} \le \cn{W} \le \on{W},
\]
so
\[
 |i(\uW)| \le \cn{\uW} \le \on{\uW}.
\]
As in~\cite{clustering}, the quantity $i(\uW)$ will play a key role in various
approximation arguments; the inequality $|i(\uW)|\le \cn{\uW}$ is key
to making these arguments work here.

Given a hyperkernel $\kf$ and a measure-preserving bijection $\tau:\sss\to\sss$,
let $\kf^\ta=(\ka^\ta_r)_{r\ge 2}$ be the hyperkernel defined by
\[
 \ka^\ta_r(x_1,\ldots,x_r) = \ka_r(\tau(x_1),\ldots,\tau(x_r)).
\]
We call a $\kf^\ta$ a {\em rearrangement} of $\kf$, and write
$\kf'\sim\kf$ if $\kf'$ is a rearrangement of $\kf$.
The cut metric extends to hyperkernels on $\oi$ as follows:
\[
 \dcut(\kf,\kf') = \inf_{\kf''\sim \kf'} \cn{\kf-\kf''}.
\]
For hyperkernels on general probability spaces, which need not be
the same, we use couplings to define $\dcut$.

Turning to graphs, our next aim is to define an extension of the
random graph $G(A_n)$.

By an {\em $n$-by-$n$ hypermatrix} $H_n$ we mean a sequence $(H_{n,r})_{r\ge 2}$ where
each $H_{n,r}$ is an $r$-dimensional array with 
entries $\hir\ge 0$, $1\le i_1,\ldots,i_r\le n$,
that is symmetric under all permutations of the coordinates.
There is a hyperkernel $\kf=\kf(H_n)=(\ka_r)_{r\ge 2}$
naturally associated to a hypermatrix $H_n$:
each $\ka_r$ is a piecewise constant function
on $[0,1]^r$ whose value on a certain hypercube of side $1/n$ is given by the appropriate
entry of $H_{n,r}$. 

Turning to the random hypergraph, as in~\cite{clustering}, the natural normalization
in the hypergraph case is unfortunately not the same as in the graph case.
Roughly speaking, for each entry $\hir$ of each $H_{n,r}$, we shall add
a hyperedge on the corresponding vertices to our hypergraph
with probability $\hir/n^{r-1}$.
Unfortunately this means that the probability that a particular
$r$-vertex hyperedge is present is then (roughly) $r!\hir/n^{r-1}$, and in particular
$2h_{ij}/n$ in the graph case.

Formally, given a hypermatrix $H_n$, let $\HH(H_n)$ be the random hypergraph
on $[n]$ in which edges are present independently,
and for any $2\le r\le n$ and $i_1<i_2<\cdots<i_r$, the probability
that the hyperedge $i_1i_2\cdots i_r$ is present is
\[
 \min\{ r! \hir/n^{r-1} , 1\}.
\]
Alternatively, it is often to convenient to consider the {\em Poisson multi-hypergraph}
version of $\HH(H_n)$: here the number of copies of a hyperedge
$i_1i_2\cdots i_r$ is simply Poisson with mean $r! \hir/n^{r-1}$, and these numbers
are independent for different hyperedges.

Turning to the graph, let $G(H_n)$ be the simple graph underlying $\HH(H_n)$,
obtained by replacing each $r$-vertex hyperedge by a complete graph
on $r$ vertices, and replacing any multiple edges by single edges.
In the Poisson multi-hypergraph variant, we keep multiple edges.

\begin{remark}\label{Rdiag}
We call an entry $\hir$ of some $H_{n,r}$ {\em diagonal} if
$i_k=i_\ell$ for some $k\ne \ell$. Note that in the definitions
of $\HH(H_n)$ and $G(H_n)$, such entries play no role. We shall see later
that, as in the graph case, convergence of $(H_n)$ to $\kf$ in $\dcut$
is unaffected by setting all diagonal entries to $0$, so (once
we have shown this), we may assume without loss of generality
that all diagonal entries are $0$. However, we do {\em not}
impose this as a condition of our results, since there is no need to do so.
\end{remark}

Given a hyperkernel $\kf$,
let $\bp_\kf$ be the compound Poisson Galton--Watson branching process
associated to $\kf$; for the formal definition see~\cite{clustering}.
We write $\rho(\kf)$ for the survival probability of $\bp_\kf$.

As in~\cite{clustering},
let $\kae$ be the {\em edge kernel} corresponding to $\kf=(\ka_r)$, defined by
\begin{equation}\label{ce}
  \kae(x,y) = \sum_{r\ge 2}r(r-1)  \int_{\sss^{r-2}}\ka_r(x,y,x_3,x_4,\ldots,x_r)
   \dd\mu(x_3)\cdots\dd\mu(x_r).
\end{equation}
Note that $\kae$ may be viewed as a (rescaled) 2-dimensional marginal
of the hyperkernel $\kf$.
As in~\cite{clustering}, a hyperkernel $\kf$ is {\em irreducible}
if the corresponding edge kernel is irreducible.
The natural extension of \refT{th1} to hyperkernels is as follows.

\begin{theorem}\label{th1h}
Let $\kf$ be an irreducible, integrable hyperkernel and $(H_n)$
a sequence of hypermatrices such that $\dcut(H_n,\kf)\to 0$.
Then $C_1(G(H_n))/n \pto \rho(\kf)$, and $C_2(G(H_n))=\op(n)$.
\end{theorem}

Arguing as in the proof of \refL{l_iid},
one can show that \refT{th1h} extends the
corresponding result of~\cite{clustering}.

In \refT{th1h} we define $\dcut$ using $\cnone{\cdot}$ for the cut norm. 
Since $\cnone{\cdot}\le \cntwo{\cdot}$, the corresponding result
for the more natural definition using $\cntwo{\cdot}$ follows immediately.

The heart of the proof of \refT{th1h} will be \refL{lh1} below, showing
that under an additional assumption, the number of vertices in components
of each fixed size is `what it should be'. Later we shall first
remove the additional assumption, and then pass from `large' components
to a single giant component.

We say that a hyperkernel $\kf=(\ka_r)$ is {\em $R$-bounded}
if $\ka_r$ is zero
for $r>R$, in which case we shall often speak of the hyperkernel $\kf=(\ka_r)_{r=2}^R$.
Correspondingly, a hypermatrix $H_n=(H_{n,r})_{r\ge 2}$ is {\em $R$-bounded}
if $H_{n,r}$ is the zero matrix for $r>R$.

As in~\cite{clustering}, we write $\rho_k(\kf)$ for the probability
that the branching process $\bp_{\kf}$ consists of $k$ particles in total.
Recall that $N_k(G)$ denotes the number of vertices of a graph
$G$ in components of order $k$.

\begin{lemma}\label{lh1}
Let $R\ge 2$ be fixed.
Suppose that $\kf$ is an $R$-bounded hyperkernel and $(H_n)$ is
a sequence of $R$-bounded hypermatrices such that $\dcut(H_n,\kf)\to 0$.
Then for each $k\ge 1$ we have
$N_k(G(H_n))/n\pto \rho_k(\kf)$.
\end{lemma}
The proof of this lemma will take up the next several subsections.
The deduction of \refT{th1h} will then be relatively easy.

\subsection{Eliminating large edge probabilities}

Given a hypermatrix $H_n$, for $r\ge 2$
let $A_{n,r}$ be the matrix with entries
\begin{equation}\label{Anrdef}
 a^{(r)}_{ij} =  n^{-(r-2)}\sum_{i_3}\sum_{i_4}\cdots\sum_{i_r} h_{iji_3i_4\ldots i_r},
\end{equation}
and let
\begin{equation}\label{Andef}
 A_n=\sum_{r\ge 2} r(r-1) A_{n,r}
\end{equation}
be the {\em marginal matrix} corresponding to $H_n$, with entries $a_{ij}$.
Note that the kernel $\ka_{A_n}$ defined from $A_n$ is simply
the edge kernel $\kae$ corresponding to $\kf(H_n)$.
Also, in the Poisson multi-graph form of our model,
if all diagonal entries are zero, then
the expected number of $ij$ edges in $G(H_n)$ is exactly $a_{ij}/n$.
(See \refR{Rdiag}.)

Given $W_r\in L^1(\sss^r)$, let $\hW_r$ be its marginal with respect to the first
two coordinates, defined by
\[
 \hW_r(x,y) = \int_{\sss^{r-2}} W_r(x,y,x_3,\ldots,x_r)\dd\mu(x_3)\cdots\dd\mu(x_r).
\]
Note that
\begin{equation}\label{cmarg}
 \cn{\hW_r} \le \cn{W_r}.
\end{equation}
Indeed, to see this simply take $S_3,\ldots,S_r=\sss$ in \eqref{cnr1}, or $f_3,\ldots,f_r=1$ in \eqref{cnr2}.

An immediate consequence is the following lemma.

\begin{lemma}\label{edgek}
Let $R\ge 2$ be fixed, and suppose that $(H_n)$ is a sequence of $R$-bounded
hypermatrices and $\kf$ an $R$-bounded hyperkernel with $\dcut(H_n,\kf)\to 0$.
Then $\dcut(A_n,\kae)\to 0$, where $A_n$ is the marginal matrix of $H_n$,
and $\kae$ is the edge kernel of $\kf$.
\end{lemma}
\begin{proof}
By definition of $\dcut$, there are measure-preserving bijections $\rn:\sss\to\sss$
such that $\cn{\kf(H_n)-\kf^{(\rn)}}\to 0$. With $\kf=(\ka_r)_{r=2}^R$,
writing $\ka_r'$ for the $r$-kernel
corresponding to $H_{n,r}$, this says exactly that $\sum_{r=2}^Rr \cn{\ka_r'-\ka_r^{(\rn)}}\to 0$.
Using \eqref{cmarg}, and noting that taking marginals commutes
with rearrangement, it follows that
$\sum_{r=2}^R r\cn{\ka_{A_{n,r}}-\hka_r^{(\rn)}}\to 0$.
Since $\cn{\cdot}$ is a norm on $L^1(\sss^2)$, we have
\[
 \cn{\ka_{A_n}-\kae^{(\rn)}}\le \sum_{r=2}^R r(r-1)\cn{\ka_{A_{n,r}}-\hka_r^{(\rn)}}\to 0,
\]
since changing the factor $r$ to $r(r-1)$ does not affect convergence to zero.
Hence $\dcut(A_n,\kae)\to 0$.
\end{proof}

\begin{remark}
To obtain a result analogous to \eqref{edgek} without the $R$-boundedness assumption,
we would have to redefine $\dcut$ for hyperkernels, replacing the
factor $r$ in \eqref{cnrs} by a factor $r(r-1)$, and only considering
`edge-integrable' limits $\kf$, i.e., hyperkernels with $\sum_r r(r-1)\int \ka_r$ finite.
\end{remark}

Let us call a sequence $(H_n)$ of hypermatrices {\em well behaved}
if two conditions hold: every diagonal entry is zero,
and $\max A_n/n\to 0$ as $n\to\infty$, where $\max A_n$ is the largest
entry of the $n$-by-$n$ marginal matrix $A_n$ corresponding to $H_n$.
Note that if $(H_n)$ is well behaved, then the probability that some
particular edge $ij$ is present in $G(H_n)$ is $o(1)$ as $n\to\infty$,
where the bound is uniform over edges.

\begin{lemma}\label{wlogh}
Let $R\ge 2$ be fixed, and suppose that $(H_n)$ is a sequence of $R$-bounded hypermatrices
and $\kf$ is an $R$-bounded hyperkernel
with $\dcut(H_n,\kf)\to 0$.
Then there is a sequence of well-behaved $R$-bounded hypermatrices $(H_n')$
such that $\on{\kf(H_n)-\kf(H_n')}\to 0$ and $\dcut(H_n',\kf)\to 0$.
\end{lemma}
\begin{proof}
Let $A_n$ be the marginal matrix corresponding to $H_n$ and let $\kae$
the edge kernel corresponding to $\kf$.
Then by \refL{edgek} we have $\dcut(A_n,\kae)\to 0$.
By \refL{l_on} there is a function $M(n)$ with $M(n)=o(n)$
such that only $o(n)$ entries of $A_n$ exceed $M(n)$, and the sum of these
entries is $o(n^2)$.
This immediately implies that the sum of any $n$ entries of $A_n$ is
$o(n^2)$.

Call an entry $a_{ij}$ of $A_n$ {\em bad} if either $a_{ij}>M(n)$ or $i=j$.
Let $S$ be the sum of the bad entries, so $S=o(n^2)$.
To define $H_n'$, simply modify $H_n$ by setting to $0$
any entry $\hir$ of $H_{n,r}$ such that $a_{i_ki_\ell}$ is bad for
some pair $i_k$, $i_\ell$,
$k< \ell$.
(In other words, we replace all entries contributing
to bad entries $a_{ij}$ in the marginal by zero.)
Then $H_n'$ is a hypermatrix, and
its marginal $A_n'=(a_{ij}')$ satisfies $a_{ij}'\le a_{ij}$
with $a_{ij}'=0$ whenever $a_{ij}$ is bad. Thus $(H_n')$ is well behaved.

Finally, for each $r$, we may think of modifying $H_{n,r}$ to obtain
$H_{n,r}'$ in $\binom{r}{2}$ stages, in each one fixing $k$ and $\ell$
and setting to zero entries $\hir$ for which $a_{i_ki_\ell}$ is bad.
The sum of the entries set to zero at each stage
is at most $n^{r-2}S$.
It follows easily that
\[
 \on{\kf(H_n)-\kf(H_n')}\le \sum_{r=2}^R \binom{r}{2}Sn^{-2} = O(S/n^2) = o(1).
\]
The final statement follows immediately, 
since
\[
 \dcut(H_n,H_n') = \dcut(\kf(H_n),\kf(H_n')) \le \cn{\kf(H_n)-\kf(H_n')}
  \le \on{\kf(H_n)-\kf(H_n')}.
\]
\end{proof}

An immediate consequence of \refL{wlogh} is the following rather informally
worded corollary.
\begin{corollary}\label{cwb}
In proving \refL{lh1}, we may assume that $(H_n)$ is well behaved.
\end{corollary}
\begin{proof}
Let $(H_n)$ and $\kf$ satisfy the assumption of \refL{lh1}, and
define $(H_n')$ as in \refL{wlogh}. Let $G_n'=G(H_n')$ and $G_n=G(H_n)$.
There is a natural coupling of $\HH(H_n')$ and $\HH(H_n)$ in which the expected
number of $r$-vertex 
hyperedges in the symmetric difference is at most $n\on{\ka_{H_{n,r}'}-\ka_{H_{n,r}}}$
(with equality if all diagonal entries are zero, at least in the
Poisson multi-hypergraph version);
by \refL{wlogh} this number is $o(n)$.
Since each hyperedge has at most $R$ vertices, and so
contributes at most $\binom{R}{2}=O(1)$ edges, summing over $2\le r\le R$ we have
$\E |E(G_n')\setdiff E(G_n)| =o(n)$.

Now $\dcut(H_n',\kf)\to 0$, so if \refL{lh1} holds in the well-behaved case, then
$N_k(G_n')/n\pto \rho_k(\kf)$.
Since adding or deleting an edge to a graph $G$ changes the number of vertices
in components of order $k$ by at most $2k$, we have
$\E |N_k(G_n)-N_k(G_n')| =o(n)$, so $N_k(G_n)/n\pto \rho_k(\kf)$ follows.
\end{proof}

\subsection{Hypertree integrals}\label{hti}

Throughout this subsection, we fix an integer $R\ge 2$. All hyperkernels
will be $R$-bounded, and all edges of all hypergraphs will have
size at most $R$.

A {\em hypertree} is simply a connected
hypergraph containing no cycles, or,
equivalently, a connected hypergraph $\HH$ in which
$|\HH|=1+ \sum(|E_i|-1)$, where the sum runs over all edges $E_i$ of $\HH$.

Given a hyperkernel $\kf=(\ka_r)_{r\ge 2}$ and a hypertree $\HH$,
we shall define $\tx(\HH,\kf)$ in analogy with \eqref{tx}.
Unfortunately, there is a difference in the normalization,
and the marginals need some further explanation.
For the latter, given $W_r\in L^1(\sss^r)$, let
\[
 \la_{W_r}(x) = \la_{W_r}^{(1)}(x) = \int_{\sss^{r-1}}W_r(x,x_2,\ldots,x_r) \dd\mu(x_2)\cdots\dd\mu(x_r).
\]
The marginal $\la_{W_r}^{(i)}$
of $W_r$ with respect to the $i$th coordinate is defined similarly.

Given $\kf=(\ka_r)_{r=2}^R$, let
\begin{equation}\label{hmdef}
 \la(x)=\la_{\kf}(x) = \sum_r r \la_{\ka_r}(x).
\end{equation}
The reason for the extra factor $r$ is that,
as noted earlier,
we essentially add a hyperedge on each ordered $r$-tuple
$v_1,\ldots,v_r$ with a probability $\ka_r/n^{r-1}$, and because
a particular vertex could appear in $r$ places in the ordered $r$-tuple,
it is then $\la(x)$ that gives the expected number of hyperedges
containing a given vertex.

We now define $\tx(\HH,\kf)$ as an integral over $\sss^{|\HH|}$ with one variable
$x_i$ for each vertex $i$ of $\HH$. The integrand has a factor
$r!\ka_r(x_{i_1},\ldots,x_{i_r})$ for
each $r$-element hyperedge $E=i_1i_2\ldots i_r$ of $\HH$,
and a factor $e^{-\la_{\kf}(x_i)}$  for each $i$

With this definition, \refT{th_FW} extends to the hyperkernel context.
\begin{theorem}\label{th_FWh}
Let $R\ge 2$ be fixed, and let $\HH$ be a hypertree in which
each hyperedge has at most $R$ elements.
Then $\kf\mapsto \tx(\HH,\kf)$ is a  bounded map on the space $\cwsR$ of $R$-bounded
hyperkernels and is
Lipschitz continuous in the cut norm. In other words, there exists a
constant $C$ (depending on $R$ and $\HH$ only) such that $\tx(\HH,\kf)\le C$ for all
$\kf\in\cwsR$, and $|\tx(\HH,\kf)-\tx(\HH,\kf')|\le C \cn{\kf-\kf'}$ for all
$\kf,\kf'\in\cwsR$.
\end{theorem}

Rather than give a formal proof, we shall briefly describe the modifications needed
to the arguments in \refSS{ss_ti}.
Note that we make take $\cn{\cdot}=\cnone{\cdot}$ or $\cn{\cdot}=\cntwo{\cdot}$ in \refT{th_FWh};
on $R$-bounded hyperkernels, these norms are equivalent.
As in \refSS{ss_ti}, in this subsection we use the norm $\cntwo{\cdot}$.

Firstly, note that \refL{L0} extends immediately: if $W_r,W_r'\in
L^1(\sss^r)$, then
\begin{equation}\label{hmc}
  \on{\la_{W_r}-\la_{W_r'}} \le \cn{W_r-W_r'}.
\end{equation}
(Perhaps the nicest way to see this is to note that, generalizing
\eqref{cmarg} in the natural way, the cut norm of any $d$-dimensional
marginal of some $W\in L^1(\sss_r)$ is at most $\cn{W}$, and that on
$L^1(\sss)$, the $L^1$ norm and cut norm coincide.)

Fix $\HH$. Extending \eqref{ty}, suppose that for each $r$-element hyperedge
$E$ of $\HH$ we have a $W_E\in\cwr$,
where $\cwr$ is the set of (not necessarily symmetric) non-negative functions $W_r\in L^1(\sss^r)$.
Then we may define $t_0(\HH,(W_E)_{E\in E(\HH)})$ in analogy with \eqref{ty},
again without the exponential factors in $\tx(\HH,\kf)$.
To reintroduce these, given any $W_r\in \cwr$ and $\ba=(a_1,\ldots,a_r)$ with
each $a_i\ge 0$, set
\[
 W_r^{\ba}(x_1,\ldots,x_r) = W_r(x_1,\ldots,x_r)\prod_{i=1}^r \exp\bigpar{-a_i \la_{W_r}^{(i)}(x_i) },
\]
in analogy with \eqref{Wabdef}.

The proof of \refL{L1} extends {\em mutatis mutandis} to give the following result.
\begin{lemma}\label{L1h}
For every fixed $\ba\ge0$, the map $W\mapsto W^{\ba}$ is Lipschitz continuous
on $\cwr$ in the cut norm; more precisely,
\begin{equation*}
\cn{W_1^{\ba}-W_2^{\ba}}\le (2^r+r2^r/e)
\cn{W_1-W_2}  
\end{equation*}
for all $W_1,W_2\in\cwr$. Also, for every
$W\in\cwr$, the $i$th marginal of $W^{\ba}$ is bounded by $e^{-1}/a_i$. \noproof
\end{lemma}
As before, the first $2^r$ can be replaced by $1$, but we do not care about the constant.

There is one minor additional complication not present in the graph case, which
we now describe.
Given a hyperkernel $\kf=(\ka_r)_{r=2}^R$, for each hyperedge $E$ of $\HH$ with $r$ vertices
define $W_E\in \cwr$ by
\begin{equation}\label{WEdef}
 W_E(x_1,\ldots,x_r) = \ka_r(x_1,\ldots,x_r) \prod_{i=1}^r \exp\bb{-\la_{\kf}(x) /d_i},
\end{equation}
where $d_i$ is the degree in $\HH$ of the $i$th vertex of $E$ (in some arbitrary ordering).
Then we have
\begin{equation}\label{txtyh}
 \tx(\HH,\kf) = \ty(\HH, (W_E)_{E\in E(\HH)}),
\end{equation}
corresponding to \eqref{txty}.
In the graph case we simply had $W_{ij}=\ka^{(1/d_i,1/d_j)}$, but this no longer holds,
since the marginals appearing in \eqref{WEdef} are those of $\kf$, not simply
those of the kernel $\ka_r$ appropriate for $r$-element hyperedges.
The extra complication is dealt with by \refL{L1ha} below.

Given $\BB>0$, let $\cwrA$ be the set of $W\in \cwr$ with all marginals bounded by $\BB$.
If $f\in L^1(\sss)$ and $W\in \cwr$, define $fW$ by
\[
 (fW)(x_1,\ldots,x_r) = f(x_1)W(x_1,\ldots,x_r).
\]
Suppose that $W\in \cwrA$ and $f_1,f_2\in L^1(\sss)$.
Then 
\begin{equation}\label{fw}
 \cn{(f_1-f_2)W} \le
 \on{(f_1-f_2)W} = \on{(f_1-f_2)\la} \le \BB\on{(f_1-f_2)},
\end{equation}
where $\la$ is the first marginal of $W$.
Now suppose that $f_1,\ldots,f_r, f_1',\ldots,f_r'\in L^1(\sss)$ with $\sn{f_i},\sn{f_i'}\le 1$
for each $i$, and that $W$, $W'\in \cwrA$.
Defining $f_1\cdots f_rW$ and $f_1'\cdots f_r'W'$ in the obvious way, we 
have
\begin{equation}\label{frw}
 \cn{(f_1\cdots f_rW)  - (f_1'\cdots f_r'W')} \le \cn{W-W'}+ \BB\sum_{i=1}^r \on{f_i-f_i'}.
\end{equation}
Indeed, we may write the difference as $(f_1\cdots f_r)(W-W')$ plus $r$ terms whose cut norms
may be bounded by \eqref{fw}; the cut norm of the first term is at most $\cn{W-W'}$
by the analogue of \eqref{cutnorm2}.

With $\HH$ fixed, let $\BB=\Delta(\HH)/e$.

\begin{lemma}\label{L1ha}
For each hyperedge $E$ of $\HH$, the map $\kf\mapsto W_E$ is Lipschitz continuous
with respect to the cut norm, and $W_E$ belongs to $\cwrA$.
\end{lemma}
\begin{proof}
Let $r$ be the number of vertices in $E$, and let $\kf=(\ka_s)_{s=2}^R$.
Let $\tW_E= \ka_r^{\ba}$, where $\ba=(r/d_1,\dots,r/d_r)$.
Since each $\ka_s$ is symmetric, all its marginals are equal; we
write $\la_s$ for any of these marginals.
Then $W_E = f_1\cdots f_r \tW_E$, where
\[
 f_i(x_i) = \exp\bb{-\la_{\kf}(x_i)/d_i +r\la_r(x_i)/d_i} = \exp\Bigpar{-\sum_{s\ne r}s\la_s(x_i)/d_i}.
\]
Since all marginals $\la_s$ are non-negative, we have $0<f_i(x)\le 1$.
Applying \refL{L1h} to $\ka_r$ tells us that $\tW_E\in \cwrA$, and that the map $\kf\mapsto \tW_E$
is Lipschitz continuous. Summing \eqref{hmc} over $2\le s\le R$, $s\ne r$, tells
us that each $f_i$ varies continuously (in $L^1$) with $\kf$,
and Lipschitz continuity of $\kf\mapsto W_E$ then follows from \eqref{frw}.
Finally, $\tW_E\in \cwrA$ and $0<f_i\le 1$ for each $i$ trivially implies $W_E\in\cwrA$.
\end{proof}

In the light of \eqref{txtyh} and \refL{L1ha}, it remains only to prove
an analogue of \refL{L2}, showing that $t_0(\HH,(W_E)_{E\in \HH})$
is Lipschitz continuous with respect to the cut norm when we assume
that each $W_E\in \cwrA$. The proofs of \refL{L2root} and \refL{L2} carry
over with trivial modifications, noting that for the latter
when we delete a single hyperedge $E$ with $r$ vertices, our hypertree
splits into $r$ hypertrees (some of which may be trivial).

\subsection{Small components}

With the preparation above behind us, the argument of~\refSS{ss_s} goes through
easily. Let us comment very briefly on the changes.
Firstly, it is more convenient in this subsection to consider hypergraphs throughout.

Given a hypergraph $\HH$, we write $N_k(\HH)$
for the number of vertices in components of order $k$,
$\nkt(\HH)$ for the number in tree components
of order $k$, and $\nkc(\HH)$ for the number in non-tree components.

The proof of \refL{l_Nkp} carries over easily to give the following result.
\begin{lemma}\label{l_Nkph}
Let $(H_n)$ be a  well-behaved $R$-bounded sequence of hypermatrices,
and $\HH_n=\HH(H_n)$ the corresponding random (Poisson multi-)hypergraphs.
Then for any fixed $k$ we have $\E \nkc(\HH_n) = o(n)$.
\end{lemma}
\begin{proof}
As in the graph case, we consider the number $M_{\le k}(\HH)$
of components of a hypergraph $\HH$ that contain a cycle and have
at most $k$ vertices. Since $\nkc(\HH_n)\le kM_{\le k}(\HH)$, it suffices
to prove that $\E M_{\le k}(\HH_n)=o(n)$.

When adding a hyperedge $E$ to a hypergraph $\HH$,
the quantity $M_{\le k}$ can increase only if $E$ creates a cycle, i.e., contains
at least two vertices $i$ and $j$ from some component $C$ of $\HH$,
and after adding $\HH$, the component containing $E$ has order at most 
$k$. This certainly implies that $E$ contains a pair $\{i,j\}$ of
distinct vertices from some component of order at most $k$.
The rest of the proof follows that of \refL{l_Nkp},
using the fact that $(H_n)$ well behaved guarantees that the expected
number of edges of $\HH_n$ containing a particular
pair $\{i,j\}$ of vertices is $o(1)$,
uniformly in $i$ and $j$.
\end{proof}

The remaining arguments in \refSS{ss_s} carry over easily.
\begin{proof}[Proof of \refL{lh1}]
Let $(H_n)$ be a sequence of $R$-bounded hypermatrices converging in $\dcut$
to an $R$-bounded hyperkernel $\kf$.
By \refC{cwb} we may assume that $(H_n)$ is well behaved.

Given a hyperedge $E=i_1\ldots i_r$ with vertices contained in $[n]$,
let $h_E=h_{i_1\ldots i_r}$ be the corresponding entry of $H_{n,r}$,
and $\mu_E=r! h_E n^{-(r-1)}$
the expected number of copies of $E$ in $\HH_n=\HH(H_n)$.
Given a connected simple hypergraph
$\F$
on $[k]$ and a sequence $\bv=(v_1,\ldots,v_k)$
of vertices of $\HH_n$, for each hyperedge $E=i_1\ldots i_r$
of $\F$ let $\bv(E)=v_{i_1}\ldots v_{i_r}$ be the image of $E$ 
under the map $i\mapsto v_i$.

As before, for a good sequence $\bv$,
let $p_{\bv}(\F)=p_{\bv}(\F,H_n)$ be the probability that the image
of $\F$ under $i\mapsto v_i$ is present in $\HH_n$, and forms a component
of $\HH_n$.
Thus
\[
 p_{\bv}(\F) =  \prod_{E\in E(\F)} \mu_{\bv(E)} \prod_{E\in E_0} \exp(-\mu_E),
\]
where $E_0$ is the set of all potential edges of $\HH_n$ that share at least one
vertex with $\{v_1,\ldots,v_k\}$.
For any $\bv$, set
\[
 p^0_{\bv}(\F) =  \prod_{E\in E(\F)} \mu_{\bv(E)} \prod_{i=1}^k \exp(-\la_n(v_i)),
\]
where $\la_n(v)$ is the sum of the probabilities of all hyperedges meeting
$v$. Note that $\la_n$ is exactly the marginal of the hyperkernel corresponding
to $H_n$, but here viewed as a function on $[n]$ rather than on $[0,1]$.

If $\bv$ is good, the only difference between
$p^0_{\bv}(\F)$ and $p_{\bv}(\F)$
is that for each $E\in E_0$ sharing $s\ge 2$ vertices with $\{v_1,\ldots,v_k\}$,
the factor $\exp(-\mu_E)$ appears $s$ times in $p^0_{\bv}(\F)$ but only once
in $p_{\bv}(\F)$. Since $(H_n)$ is well behaved, for any $i\ne j$
the sum of $\mu_E$ over hyperedges $E$ containing both $i$ and $j$ is $o(1)$,
so it follows as before that $ p^0_{\bv}(\F) \sim  p_{\bv}(\F)$.

Let $\T$ 
be a hypertree. Summing $p^0_{\bv}(\T)$ over {\em all} sequences
$\bv$ we obtain exactly $n \tx(\T,\kf)$. 
The rest of the proof of \refL{l_Nk} goes through essentially unchanged to show that
the contribution from bad sequences $\bv$ is negligible, and summing over
hypertrees $\T$, and using \refL{l_Nkph}, it follows that
$\E N_k(\HH_n)/n \to\rho_k(\kf)$.
(Note that \eqref{xk=t} holds unchanged for hypergraphs too, with the
normalizations used here.)
As before, considering disjoint copies of two trees gives convergence in probability,
as required.
\end{proof}

Finally, we note that the result we have just proved extends from $R$-bounded hyperkernels
to general hyperkernels.

\begin{corollary}\label{c_nk}
Let $\kf$ be an integrable hyperkernel and $(H_n)$ a sequence
of hypermatrices with $\dcut(H_n,\kf)\to 0$, and set $G_n=G(H_n)$.
Then $N_k(G_n)/n\pto \rho_k(\kf)$.
\end{corollary}
\begin{proof}
Firstly, it makes no difference whether we work with the hypergraphs
$\HH_n=\HH(H_n)$ or the underlying
graphs $G_n=G(H_n)$, as these have
exactly the same components.

Fix $k\ge1$.
Let $\kf=(\ka_r)_{r\ge 2}$. For $R\ge2$, set $\kf^R=(\ka_r)_{r=2}^R$,
and similarly define $H_n^R$ by omitting all matrices $H_{n,r}$ with $r> R$.
Fix $\eps>0$.
Since $\kf$ is integrable, we have $i(\kf^R)\upto i(\kf)$ as $R\to\infty$.
By Theorem 2.13(i) of~\cite{clustering}, we have $\rho_k(\kf^R)\to \rho_k(\kf)$.
Hence there is some $R$ such that $i(\kf-\kf^R)\le\eps$ and
\begin{equation}\label{close}
 |\rho_k(\kf)-\rho_k(\kf^R)|\le \eps.
\end{equation}
Fix such an $R$.
From the definition of $\dcut$, we have 
\begin{align*}
 i\bigpar{\kf(H_n)-\kf(H_n^R)} & \le i(\kf-\kf^R)+\dcut\bigpar{\kf(H_n)-\kf(H_n^R),\kf-\kf^R} \\
   &\le \eps +\dcut(\kf(H_n),\kf) = \eps+o(1).
\end{align*}
Coupling $\HH_n$ and $\HH_n^R=\HH(H_n^R)$ in the natural way so that the former contains
the latter, the expected sum of the sizes of the extra hyperedges in $\HH_n$
is at most $n i\bigpar{\kf(H_n)-\kf(H_n^R)}\le (\eps+o(1))n$.
Since adding a clique of size $r$ to a graph $G$ changes the number of vertices
in components of size at most $k$ by at most $rk$, it follows that for $k$ fixed we have
$\E | N_k(\HH_n)- N_k(\HH_n^R) | \le k\eps n+o(n)$, so for $n$ large enough,
\begin{equation}\label{nkhr}
 \Pr\bigpar{ | N_k(\HH_n)- N_k(\HH_n^R) | \ge k\sqrt{\eps} } \le 2\sqrt{\eps},
\end{equation}
say.
Applying \refL{lh1} to the sequence $(H_n^R)$, we have $N_k(\HH_n^R)=\rho_k(\kf^R)+\op(n)$.
Since $\eps>0$ was arbitrary, the result follows
from this, \eqref{close} and \eqref{nkhr}.
\end{proof}

\subsection{Proof of \refT{th1h}}

We have just seen that for each $k$ we have the `right' number of vertices
of $G(H_n)$ in components of order $k$; it remains only to show,
using the additional assumption of irreducibility, that almost all vertices
in large components in fact form a single giant component.

\begin{proof}[Proof of \refT{th1h}.]
As usual, \refC{c_nk} implies that there is some $\omega=\omega(n)\to\infty$,
which we may take to be $o(n)$, such that
\begin{equation}\label{hbig}
 N_{\ge \omega}(G(H_n))/n \pto \rho(\kf).
\end{equation}

Let $G_n=G(H_n)$.  As in the proof of \refT{th1}, in the light of \eqref{hbig}
it suffices to show that $C_1(G_n)\ge \rho(\kf)n+\op(n)$.
In doing so we may of course assume that $\rho(\kf)>0$.

Fix $\eps>0$.
Theorem 2.12(i) of~\cite{clustering} tells us that as $\gamma\to 0$
we have $\rho((1-\gamma)\kf)\upto \rho(\kf)$, so there
is some $\gamma$ with $\rho((1-\gamma)\kf)>\rho(\kf)-\eps$.
In the Poisson multi-hypergraph form, we may write
$\HH_n=\HH(H_n)$ as $\HH_n'\cup \HH_n''$ where
 $\HH_n'=\HH((1-\gamma)H_n)$, $\HH_n''=\HH(\gamma H_n)$, and
$\HH_n'$ and $\HH_n''$ are independent.

Writing $G_n'$ for the graph corresponding to $\HH_n'$,
applying \eqref{hbig} to $(\HH_n')$
there is some $\omega=\omega(n)\to\infty$ such that
\begin{equation*}
 N_{\ge\omega}(G_n')\ge (\rho((1-\gamma)\kf)-\eps)n \ge (\rho(\kf)-2\eps) n
\end{equation*}
holds whp. We shall attempt to use the hyperedges of $\HH_n''$ 
to join up the large components of $G_n'$.

As in~\cite{clustering}, the trick is to select one edge from each hyperedge,
to obtain a graph. More precisely, let $G_n''$ be the random multi-graph obtained
from $\HH_n''$ by replacing each hyperedge $E$ of order $r$ by one 
of the $\binom{r}{2}$ corresponding edges, chosen uniformly at random.
From the Poisson nature of the model, different edges in $G_n''$ are present
independently.

Let $B_n=2\sum_{r\ge 2} A_{n,r}$, where $A_{n,r}$ is the matrix defined by
\eqref{Anrdef}. The edge probabilities in $G_n''$ are given
by $\gamma$ times the entries of $B_n$. (Note that the coefficient of $A_{n,r}$ is smaller
here than in \eqref{Andef}, by a factor $1/\binom{r}{2}$, corresponding
to choosing one out of $\binom{r}{2}$ edges.)

Let $\tau$ be the rescaled edge-kernel defined by
\begin{equation*}
  \tau(x,y) = 2\sum_{r\ge 2} \int_{\sss^{r-2}}\ka_r(x,y,x_3,x_4,\ldots,x_r)
   \dd\mu(x_3)\cdots\dd\mu(x_r),
\end{equation*}
i.e., by replacing the factor $r(r-1)$ in \eqref{ce} by a factor $2$.
Using~\eqref{cmarg} and arguing as in the proof of Lemma~\ref{edgek},
but replacing each appearance of $r(r-1)$ by $2$, it
is easy to check that $\dcut(\ka_{B_n},\tau) \to 0$; 
this time, since $2\le r$, there is no need to truncate the sums over $r$.

Now $\kf$ is irreducible by assumption, which means exactly that $\kae$ is irreducible.
Since $\kae$ and $\tau$ are non-zero in the same places, it follows that $\tau$
is irreducible.
Since the graphs $G_n''$ have the distribution $G(\gamma B_n)$, and $\dcut(B_n,\tau)\to 0$,
Lemma~\ref{l_join} tells us that given any two sets $X$ and $Y$ of $\eps n$ vertices
of $G_n''$, the probability that there is no path in $G_n''$ from $X$ to $Y$ is exponentially
small. As before we may apply this to all partitions of the large
components
of $G_n'$ into two sets each containing at least $\eps n$ vertices
to deduce that whp we have $C_1(G_n)\ge (\rho(\kf)-3\eps)n$, completing the proof.
\end{proof}

Theorem~\ref{th1h} implies a result for branching processes
corresponding to \refT{th_br}; we  
leave the details to the reader.

Finally, let us note that using the trick of selecting one edge from
each hyperedge 
above, it is very easy to extend \refT{th_stab} to the graphs $G(H_n)$
considered in \refT{th1h}.

\newcommand\arxiv[1]{\webcite{arXiv:#1.}}
\newcommand\webcite[1]{
\texttt{\def~{{\tiny$\sim$}}#1}\hfill\hfill}

\begin{ack}
Part of this research was conducted during the programme
`Combinatorics and Statistical Mechanics'
at the Isaac Newton Institute, Cambridge; we are grateful to the Institute
and to the programme organizers.
\end{ack}


\begin{thebibliography}{99}

\bibitem{Alon} N. Alon,
 A note on network reliability, in
 \emph{Discrete probability and algorithms (Minneapolis, MN, 1993)},
 IMA Vol. Math. Appl., 72, Springer, New York, 1995, pp. 11--14.

\bibitem{BCS} M.~Biskup, L.~Chayes and S.A.~Smith,
 Large-deviations/thermodynamic approach to percolation on the complete graph,
 {\em Random Structures Algorithms} {\bf 31}  (2007),  354--370. 



\bibitem{QRperc} B.~Bollob\'as, C.~Borgs, J.~Chayes and O.~Riordan,
 Percolation on dense graph sequences,
 {\em Annals of Probability} {\bf 38} (2010), 150--183.


\bibitem{kernels} B.~Bollob\'as, S.~Janson and O.~Riordan,
 The phase transition in inhomogeneous random graphs,
{\em Random Structures and Algorithms} {\bf 31} (2007), 3--122.

\bibitem{clustering} B.~Bollob\'as, S.~Janson and O.~Riordan,
 Sparse random graphs with clustering,
 to appear in {\em Random Structures and Algorithms}.
\arxiv{0807.2040}

\bibitem{BRsparse} B.~Bollob\'as and O.~Riordan,
 Metrics for sparse graphs,
 in {\em Surveys in Combinatorics 2009}, London Math. Soc. Lecture Note Series {\bf 365},
S. Huczynska, J.D. Mitchell and C.M.Roney-Dougal eds, CUP (2009), pp. 212-287.

\bibitem{BRsparse2} B.~Bollob\'as and O.~Riordan,
 Sparse graphs: metrics and random models,
 to appear in {\em Random Structures and Algorithms}. 
\arxiv{0812.2656}

\bibitem{BCL:unique} C.~Borgs, J.T.~Chayes and L.~Lov\'asz,
 Moments of two-variable functions and the uniqueness of graph limits,
 to appear in {\em Geom. Funct. Anal.}
 \arxiv{0803.1244}


\bibitem{BCLSV:1} C.~Borgs, J. T.~Chayes, L.~Lov\'asz, V. T.~S\'os and
K.~Vesztergombi, 
Convergent sequences of dense graphs I: Subgraph
frequencies, metric properties and testing,
{\em Advances in Math.} {\bf 219} (2008), 1801--1851.

\bibitem{Chung_coset} F.R.K.~Chung,
 Diameters and eigenvalues,
 {\em J. Amer. Math. Soc.} {\bf 2} (1989), 187--196.

\bibitem{CG} F.~Chung and R.~Graham,
 Sparse quasi-random graphs, 
 {\em Combinatorica} {\bf 22} (2002), 217--244.

\bibitem{CGW89} F.R.K.~Chung, R.L.~Graham and R.M.~Wilson,
 Quasi-random graphs, {\em Combinatorica} {\bf 9} (1989),  345--362.

\bibitem{SJ209}
 P. Diaconis and S. Janson,
 Graph limits and exchangeable random graphs,
 {\em Rendiconti di Matematica}
 {\bf 28} (2008), 33--61.


\bibitem{ERpolarity} P.~Erd\H os and A.~R\'enyi,
 On a problem in the theory of graphs, 
 {\em Magyar Tud. Akad. Mat. Kutat\'o Int. K\"ozl.} {\bf 7} (1962), 623--641.

\bibitem{FKquick} A.~Frieze and R.~Kannan,
 Quick approximation to matrices and applications,
 {\em Combinatorica} {\bf 19} (1999), 175--220.

\bibitem{SJ210}
 S. Janson,
 Standard representation of multivariate functions on a general
 probability space,
 {\em Electron. Commun. Probab.} {\bf  14}  (2009), 343--346. 


\bibitem{SJ212}
 S. Janson,
 Asymptotic equivalence and contiguity of some random graphs,
 {\em Random Structures and Algorithms} {\bf 36} (2010), 26--45.

\bibitem{JLR}
 S. Janson, T. \L uczak and A. Ruci\'nski,
 \emph{Random Graphs}.
 Wiley, New York, 2000.

\bibitem{LSz1}
 L.~Lov\'asz and B.~Szegedy,
 Limits of dense graph sequences,
 {\em J. Combin. Theory B} {\bf 96} (2006), 933--957.

\bibitem{LPS} A.~Lubotzky, R.~Phillips and P.~Sarnak,
 Ramanujan graphs,
 {\em Combinatorica} {\bf 8} (1988), 261--277. 


\bibitem{LMcD} M.J.~Luczak and C.~McDiarmid,
 Bisecting sparse random graphs,
 {\em Random Structures and Algorithms} {\bf 18} (2001), 31--38.

\bibitem{McD} C.~McDiarmid,
 On the method of bounded differences, 
 in {\em Surveys in combinatorics, 1989},
 LMS Lecture Note Series {\bf 141}, Cambridge Univ. Press (1989), pp 148--188.

\bibitem{OC} N.~O'Connell,
 Some large deviation results for sparse random graphs,
 {\em Probab. Theory Related Fields} {\bf 110} (1998), 277--285.

\bibitem{Talagrand} M.~Talagrand,
 Concentration of measure and isoperimetric inequalities in product spaces,
 {\em Inst. Hautes \'Etudes Sci. Publ. Math.} {\bf 81} (1995), 73--205. 

\bibitem{Tho} A. Thomason, Pseudo-random graphs, in
 {\em Proceedings of Random Graphs} (M. Karonski, ed.),  Pozna\'n,
 1985, {\em Annals of Discrete Mathematics}, {\bf 33} (1987)
 307--331.

\bibitem{Tho2} A.~Thomason, 
 Random graphs, strongly regular graphs and pseudorandom graphs,
 in {\em Surveys in Combinatorics 1987 (New Cross, 1987)},
  London Math. Soc. Lecture Note Ser., {\bf 123},
 Cambridge Univ. Press, Cambridge  (1987), pp 173--195.

\end{thebibliography}
\end{document}